\documentclass{amsart}[12pt]
\usepackage{mathrsfs}
\usepackage{amssymb}
\usepackage{amsfonts}
\usepackage{amsmath,color, esvect,autobreak}
\usepackage[colorlinks,linkcolor=blue,urlcolor=cyan,citecolor=blue]{hyperref}

\usepackage{genyoungtabtikz}

\newcommand\reduline{\bgroup\markoverwith
      {\textcolor{red}{\rule[0.5ex]{2pt}{1pt}}}\ULon}

\usepackage{comment, amsmath, amssymb, amsgen, amsthm, amscd, xspace, color, epsfig, float, epic}

\newcommand{\Rmnum}[1]{\expandafter\@slowromancap\romannumeral #1@}

\newcommand{\lam}{\lambda}

\newcommand{\hdom}[3]{\draw (0+#1,0-#2) rectangle (2+#1,-1-#2)++(-1,+0.5) node {$ #3$};}
\newcommand{\vdom}[3]{\draw (0+#1,0-#2) rectangle (1+#1,-2-#2)++(-0.5,+1) node {$ #3$};}
\newcommand{\hobox}[3]{\draw (0+#1,0-#2) rectangle (1+#1,-1-#2)++(-0.5,+0.5) node {$ #3$};}

\newcommand{\domscale}{0.51}

\newcommand{\ev}{^{\mathrm{ev}}}
\newcommand{\od}{^{\mathrm{odd}}}

\newcommand{\GK}{\mathrm{GKdim}}

\newcommand{\pf}{\begin{proof}}
\newcommand{\epf}{\end{proof}}
\newcommand{\eq}{\begin{equation}}
\newcommand{\eeq}{\end{equation}}
\newcommand{\eqn}{\begin{equation*}}
\newcommand{\eeqn}{\end{equation*}}

\newcommand{\frb}{\mathfrak{b}}

\newcommand{\frg}{\mathfrak{g}}
\newcommand{\frh}{\mathfrak{h}}

\newcommand{\frl}{\mathfrak{l}}

\newcommand{\frn}{\mathfrak{n}}

\newcommand{\frp}{\mathfrak{p}}

\newcommand{\fru}{\mathfrak{u}}

\allowdisplaybreaks
\newtheorem{theorem}[equation]{Theorem}
\newtheorem{corollary}[equation]{Corollary}
\newtheorem{proposition}[equation]{Proposition}
\newtheorem{lemma}[equation]{Lemma}

\newtheorem{remark}[equation]{Remark}
\theoremstyle{definition}
\newtheorem{definition}[equation]{Definition}

\newtheorem{example}[equation]{Example}

\numberwithin{equation}{section}
\usepackage[centertableaux]{ytableau}
\colorlet{srcol}{black!15}

\makeatletter
\pgfkeys{/ytableau/options,
  noframe/.default = false,
  noframe/.is choice,
  noframe/true/.code = {%
    \global\let\vrule@YT=\vrule@none@YT
    \global\let\hrule@YT=\hrule@none@YT
  },
  noframe/false/.code = {%
    \global\let\vrule@YT=\vrule@normal@YT
    \global\let\hrule@YT=\hrule@normal@YT
  },
  noframe/on/.style = {noframe/true},
  noframe/off/.style = {noframe/false},
}
\makeatother
\ytableausetup{noframe=off,boxsize=1.3em}
\let\ytb=\ytableaushort

\newcommand{\tytb}[1]{{\tiny\ytb{#1}}}

\makeatletter

\begin{document}

\title[ socular highest weight modules and Richardson orbits]{A characterization of socular  highest weight modules and Richardson orbits of classical types}

\author{Zhanqiang Bai}
\address[Bai]{School of Mathematical Sciences, Soochow University, Suzhou 215006, China}
\email{zqbai@suda.edu.cn}

\author{Shaoyi Zhang}
\address[Zhang]{School of Mathematical Sciences, Soochow University, Suzhou 215006, China}
\email{shaoyizhangedu@163.com}




\bigskip

\begin{abstract}
Let $\mathfrak{g}$ be a simple complex Lie algebra of classical type with a Cartan subalgebra $\mathfrak{h}$. We fix a standard parabolic subalgebra $\mathfrak{p}\supset \mathfrak{h}$. The socular simple modules are just those highest weight modules with the largest possible Gelfand-Kirillov dimension in the corresponding parabolic category $\mathcal{O}^{\mathfrak{p}}$.
In this article, we will give an explicit characterization for these modules. When the module is integral, our characterization is given by the information of the corresponding Young tableau associated to the given highest weight module. When the module is nonintegral, we still have some characterization by using the results in the integral case. In our characterization, we define a particular Young diagram called Z-diagram. From this diagram, we can describe the partition type of the unique Richardson orbit associated to the given parabolic subalgebra $\mathfrak{p}$.
\end{abstract}

%

\maketitle

%


%
\section{Introduction}
%

Let $\mathfrak{g}$ be a finite-dimensional complex semisimple Lie
algebra.
Fix a Cartan subalgebra $\mathfrak{h}$ and denote by $\Phi$ the root system associated with $(\frg, \frh)$. Choose a positive root system
$\Phi^+\subset\Phi$ and a simple system $\Delta=\{\alpha_i~ |~ 1\leq i\leq n\}\subset\Phi^+$. Let $\rho$ be the half sum of roots in $\Phi^+$. Let
$\mathfrak{g}=\bar\frn\oplus\mathfrak{h}\oplus
\mathfrak{n}$ be the Cartan decomposition of $\frg$ with nilpotent radical $\frn$ and its dual space $\bar\frn$.  Moreover,
$\frb=\frh\oplus\frn$ is the Borel subalgebra corresponding to $\Phi^+$. Choose a subset $I\subset \Delta$. Then it generates a subsystem
$\Phi_I\subset\Phi$.
Let $\frp_I$ be the standard parabolic subalgebra corresponding to $I$ with Levi decomposition $\frp_I=\frl_I\oplus\fru_I$. We frequently drop the
subscript $I$ if there is no confusion.

Let $F(\lambda)$ be a finite-dimensional irreducible $\mathfrak{l}$-module with highest weight $\lambda-\rho \in\frh^*$. It can also be viewed as a
$\mathfrak{p}$-module with a trivial $\mathfrak{u}$-action. The {\it generalized Verma module} $M_I(\lambda)$ is defined by
\[
M_I(\lambda):=U(\frg)\otimes_{U(\frp)}F(\lambda).
\]
The irreducible quotient of $M_I(\lambda)$ is denoted by $L(\lambda)$, which is a highest weight module of highest weight $\lambda-\rho$. We use $\mathcal{O}^{\mathfrak{p}}$ to denote the corresponding parabolic category.

In \cite{I}, Irving called $L(\mu)$ or $\mu$ {\it socular} if $L(\mu)$ is a summand of the socle of a generalized Verma module $M_I(\lambda)$ in the category $\mathcal{O}^{\mathfrak{p}}$. In \cite[A. 2]{I}, Irving showed that these socular highest weight modules $L(\mu)$ are just those modules with the largest possible Gelfand-Kirillov dimension $\dim(\mathfrak{u})$ in $\mathcal{O}^{\mathfrak{p}}$. These modules play an important role in the study of parabolic category $\mathcal{O}^{\mathfrak{p}}$. Many mathematicians use these modules to study the properties of  category $\mathcal{O}$ or other types of algebras, for example, see Refs. \cite{Br,BK,BLW,MC, MCb,X}. Another motivation for us to study these modules is that a scalar generalized Verma module $M_I(\lambda)$ (here `scalar' means that $\dim F(\lambda)=1$) is reducible if and only if the Gelfand-Kirillov dimension of its irreducible quotient $L(\lambda)$ is strictly less than $\dim(\fru_I)$ (see \cite{BXiao}). From \cite{BXX},  we can write $\lambda=w\mu$ for a unique anti-dominant $\mu\in \mathfrak{h}^*$ and a unique minimal length element $w\in W_{[\lambda]}$, where $W_{[\lambda]}$ is the integral Weyl group of $\lambda$.  It is known that a highest weight module $L(\lambda)$ in $\mathcal{O}^{\mathfrak{p}}$ is socular  if and only if $w$ belongs to  the Kazhdan-Lusztig right cell containing $w_0^{\mathfrak{p}}$, where $w_0^{\mathfrak{p}}$
 is the longest element in the parabolic subgroup
of $W$ corresponding to the parabolic subalgebra $\mathfrak{p}$, see \cite[Theorem 48]{MCb}. But outside of type $A$ there is no nice combinatorial description of KL right cells
(like via the Robinson-Schensted algorithm in type $A$).
 Garfinkle \cite{Gaa,Gab,Gac} used domino tableaux to describe KL right cells, which are not easy  to use.

Recently, Bai-Xie \cite{BXie} and Bai-Xiao-Xie \cite{BXX} generalized the  Robinson-Schensted algorithm and found some practical combinatorial algorithms to compute the Gelfand-Kirillov
dimension of any simple highest weight module when $\mathfrak{g}$ is a classical Lie algebra. In this article, we will use their algorithms to  give an explicit characterization for these socular highest weight modules.

Now we let $\mathfrak{g}=\mathfrak{so}(2n+1,\mathbb{C})$. We choose
$\Phi^+=\{e_i-e_j|1\leq i<j\leq n\}\cup \{e_i|1\leq i\leq n\}$ and a simple system $\Delta=\{\alpha_i:=e_i-e_{i+1}|1\leq i\leq n-1\}\cup \{\alpha_n:=e_n\}$. We choose a subset $I\subset \Delta$.  There will exist some positive integers ${n_1,n_2,\cdots,n_{k-1}}$ with $n_1+n_2+\cdots+n_{k-1}\leq n$ such that $\Delta\setminus I=\{\alpha_{p_1},\alpha_{p_2},\cdots,\alpha_{p_{k-1}}\}$, where   $p_t=\sum_{i=1}^{t}n_i$.
This subset $I$ will  generate a subsystem
$\Phi_I\subset\Phi$. Let $\frp_I$ be the standard parabolic subalgebra corresponding to $I$ with Levi decomposition $\frp_I=\frl_I\oplus\fru_I$. 
	We call  $\frp_I$ a {\it standard  parabolic subalgebra of type $(n_1, n_2,\cdots,n_k)$} with $n_k= n-\sum_{i=1}^{k-1} n_i$. Note that we may have $n_k=0$.
We can similarly define standard  parabolic subalgebras for $\mathfrak{sp}(n,\mathbb{C})$. 

For $\mathfrak{g}=\mathfrak{so}(2n,\mathbb{C})$. We choose
$\Phi^+=\{e_i\pm e_j|1\leq i<j\leq n\}$ and a simple system $\Delta=\{\alpha_i:=e_i-e_{i+1}|1\leq i\leq n-1\}\cup \{\alpha_n:=e_{n-1}+e_n\}$. Similarly we can get a  standard  parabolic subalgebra of type $(n_1, n_2,\cdots,n_k)$. 
When $n_k=1$, we can regard $\mathfrak{p}_I$ as a standard parabolic subalgebra of type $(n_1, n_2,\cdots,n_{k-2},n_{k-1}+1,0)$. These two parabolic subalgebras are isomorphic. For example, when $\Delta\setminus I=\{\alpha_{n-1}\}$, $\mathfrak{p}_I$ will be a standard parabolic subalgebra of type $(n-1,1)$. When $\Delta\setminus I=\{\alpha_{n}\}$, $\mathfrak{p}_I$ will be a standard parabolic subalgebra of type $(n,0)$. These two standard parabolic subalgebras are isomorphic.
From now on, we only consider the cases of $n_k\neq 1$ for $\mathfrak{g}=\mathfrak{so}(2n,\mathbb{C})$ in this article.

For an integral weight $\lambda\in \mathfrak{h}^*$, we write
$\lambda=(\lambda_1,\cdots,\lambda_n)$ and $$\lambda^-=(\lambda_1,\cdots,\lambda_n,-\lambda_n,\\
\cdots,-\lambda_1).$$ By using the famous R-S algorithm in \cite{BXie}, from $\lam$ we can get a Young tableau $P(\lambda)$, see \S \ref{RS}. We use $p(\lambda)=(p_1,\cdots,p_N)$ to denote its shape, which is a partition of $n$. We say $q=(q_1, \cdots, q_N)$ is the shape of the dual tableau of a Young tableau $P$ with shape $p=(p_1, \cdots, p_N)$ and write $q=p^t$ if $q_i$ is the length of $i$-th column of the Young tableau $P$. 

We recall the definition of Hollow tableaux in \cite{BXX}.
\begin{definition}
For a Young tableau $P$ with shape $p$, use $ (k,l) $ to denote the box in the $ k $-th row and the $ l $-th column.
We say the box $ (k,l) $ is \textit{even} (resp. \textit{odd}) if $ k+l $ is even (resp. odd). Let $ p_i \ev$ (resp. $ p_i\od $) be the number of even (resp. odd) boxes in the $ i $-th row of the Young tableau $ P $. 
One can easily check that \[	p_i\ev=\begin{cases}
		\left\lceil \frac{p_i}{2} \right\rceil&\text{ if } i \text{ is odd},\\
		\left\lfloor \frac{p_i}{2} \right\rfloor&\text{ if } i \text{ is even},
	\end{cases}
	\quad p_i\od=\begin{cases}
		\left\lfloor \frac{p_i}{2} \right\rfloor&\text{ if } i \text{ is odd},\\
		\left\lceil \frac{p_i}{2} \right\rceil&\text{ if } i \text{ is even}.
	\end{cases}
\]
Here for $ a\in \mathbb{R} $, $ \lfloor a \rfloor $ is the largest integer $ n $ such that $ n\leq a $, and $ \lceil a \rceil$ is the smallest integer $n$ such that $ n\geq a $. For convenience, we set
\begin{equation*}
	p\ev=(p_1\ev,p_2\ev,\cdots)\quad\mbox{and}\quad p\od=(p_1\od,p_2\od,\cdots).
\end{equation*}
\end{definition}

\begin{example}
	Let $ p=(5,5,4,3,3) $ be the shape of a Young diagram $P$. The even and odd boxes in $P$ are marked as follows:
	\[
		\tytb{EOEOE,OEOEO,EOEO,OEO,EOE}.
		\]
 
	Then $ p\ev=(3,2,2,1,2)$ and $ p\od=(2,3,2,2,1) $.
\end{example}

Now we can give our characterization of socular highest weight modules.

\begin{theorem} \label{mainA}
Let $\mathfrak{g}=\mathfrak{sl}(n,\mathbb{C})$. Suppose $\mathfrak{p}$ is a standard parabolic subalgebra of type $(n_1, n_2,\cdots,n_k)$. A simple integral highest weight module $L(\lambda)$  in $\mathcal{O}^{\mathfrak{p}}$  is socular  if and only if 
$p(\lambda)^t=(m_1,\cdots,m_k)$, where $(m_1,\cdots,m_k)$ is the  arrangement of the sequence $(n_1,n_2,..,n_{k})$ in descending order.
\end{theorem}

Since this result is well-known in the language of KL right cells \cite{MCb}, we will omit its proof  in this article.

\begin{theorem} \label{mainBC}
Let $\mathfrak{g}=\mathfrak{sp}(n,\mathbb{C})$ or $\mathfrak{so}(2n+1,\mathbb{C})$. Suppose $\mathfrak{p}$ is a standard parabolic subalgebra of type $(n_1, n_2,\cdots,n_k)$. A simple integral highest weight module $L(\lambda)$  in $\mathcal{O}^{\mathfrak{p}}$  is socular  if and only if $P(\lambda^{-})$ has the same odd boxes as a Z-diagram of type $(n_k; n_1,\cdots,n_{k-1})$. 

\end{theorem}
Here Z-diagram means a Young diagram defined in Definition \ref{zdigram}. The numbers of odd boxes and even boxes in a Z-diagram of type $(n_k; n_1,\cdots,n_{k-1})$ will be given in Corollary \ref{numberev}.


Similarly we have the following.
\begin{theorem} \label{mainD}
 Let $\mathfrak{g}=\mathfrak{so}(2n,\mathbb{C})$. $\mathfrak{p}$ is a standard parabolic subalgebra of type $(n_1, n_2,\cdots,n_k)$. When $n_k\neq 1$, a simple integral highest weight module  $L(\lambda)$  in $\mathcal{O}^{\mathfrak{p}}$  is socular  if and only if $P(\lambda^{-})$ has the same even boxes as a Z-diagram of type  $(n_k; n_1,\cdots,n_{k-1})$. When $n_k=1$, a simple integral highest weight module  $L(\lambda)$  in $\mathcal{O}^{\mathfrak{p}}$  is socular  if and only if $P(\lambda^{-})$ has the same even boxes as a Z-diagram of type  $(0; n_1,\cdots,n_{k-2}, n_{k-1}+1)$.
\end{theorem}
The second purpose of this article is to give a characterization for the partition of the Richardson orbit associated to
the parabolic subalgebra $\mathfrak{p}$. Let $G$ be a complex reductive Lie group with Lie algebra $\mathfrak{g}$. Suppose $\mathfrak{p}=\mathfrak{l}\oplus\fru$ is a standard parabolic subalgebra of type $(n_1, n_2,\cdots,n_k)$. The $G$ saturation of $\fru$ is a nilpotent orbit of $\mathfrak{g}$, which meets $\fru$  in an open dense set. Such an orbit
is  called a {\it Richardson orbit}. It is uniquely determined by $\mathfrak{p}$. They  play an
important role in the representation theory of $G$ and  attracted
considerable  attention (e.g., \cite{Ri,He,HR,Pa,Ja, Ba}).

Recall that from a partition ${\bf q}$ of $2n$ or $2n+1$, we can identify it with a Young diagram $Q$. Then after moving some boxes in $Q$, we can get another diagram, which will give us a particular partition corresponding to a nilpotent orbit of type $X$ (here $X=C,D$ or $B$). This algorithm is called the collapse of the partition ${\bf q}$, denoted by ${\bf q}_X$. Some details can be found in Definition \ref{collapse}.
For types $C$ and $D$, we find that from the partition ${\bf q}$ corresponding to the  Z-diagram of type  $(n_k; n_1,\cdots,n_{k-1})$, we can  get the Richardson orbit associated to the given parabolic subalgebra $\mathfrak{p}$ of type  $(n_1,\cdots,n_k)$ after doing the collapse of the partition ${\bf q}$, see Theorems \ref{rich-c} and \ref{rich-d}. For type $B$, we start from a modified partition which is different at one part with the partition ${\bf q}$ corresponding to the  Z-diagram of type  $(n_k; n_1,\cdots,n_{k-1})$. Finally we can also get the Richardson orbit associated to the given parabolic subalgebra $\mathfrak{p}$ of type  $(n_1,\cdots,n_k)$ after doing the $B$-collapse, see Theorem \ref{rich-b}.

The paper is organized as follows: in \S \ref{Pre}, we recall the algorithms of computing Gelfand-Kirillov dimensions of highest weight modules in \cite{BXX,BXie} and the H-algorithm (see \cite{BMW}) from which we can construct a special partition (in the sense of Lusztig \cite{Lus79}) from a domino type partition. In \S \ref{proofbc}, we prove our main theorem for types $B$ and $C$ after we define the Z-diagram for a parabolic subalgebra $\mathfrak{p}$ of type  $(n_1,\cdots,n_k)$. In \S \ref{proofd}, we prove our main theorem for type $D$. In \S \ref{nonintegral}, we give a characterization for the nonintegral socular highest weight modules.
In \S \ref{richar}, we give a characterization for the partition of the Richardson orbit associated to the parabolic subalgebra $\mathfrak{p}$ of type  $(n_1,\cdots,n_k)$.

%
%
\section{Preliminaries}\label{Pre}
%
%
Before we prove our main theorems, we first recall some useful results about Gelfand-Kirillov dimension. The details can be found in \cite{V} and \cite{BXie}.

\subsection{Gelfand-Kirillov dimension}\label{RS}
Let $M$ be a $U(\mathfrak{g})$-module generated by a finite-dimensional subspace $M_0$ (that is, $M$ is finitely generated). Let $U_{n}(\mathfrak{g})$
be
the standard filtration of $U(\mathfrak{g})$. Set $M_n=U_n(\mathfrak{g})\cdot M_0$ and
\(
\text{gr} (M)=\bigoplus\limits_{n=0}^{\infty} \text{gr}_n M,
\)
where $\text{gr}_n M=M_n/{M_{n-1}}$. Thus $\text{gr}(M)$ is a graded module of $\text{gr}(U(\mathfrak{g}))\simeq S(\mathfrak{g})$.

\begin{definition}The \textit{Gelfand-Kirillov dimension} of $M$  is defined by
	\begin{equation*}
		\operatorname{GKdim} M = \varlimsup\limits_{n\rightarrow \infty}\frac{\log\dim( U_n(\mathfrak{g})M_{0} )}{\log n}.
	\end{equation*}
\end{definition}

	It is easy to see that the above definition is independent of the choice of $M_0$.

Now  let  $\mathfrak{g}$ be a complex simple  Lie algebra.   We choose a subset $I\subset \Delta$. 
This set $I$ will  generate a subsystem
$\Phi_I\subset\Phi$. Let $\frp_I$ be the standard parabolic subalgebra corresponding to $I$ with the Levi decomposition $\frp_I=\frl_I\oplus\fru_I$.
Let $F(\lambda)$ be a finite-dimensional irreducible $\frl_I$-module with highest weight $\lambda-\rho \in\frh^*$.
The generalized Verma modules $M_I(\lambda)
=U(\frg)\otimes_{U(\frp_I)}F(\lambda)$ has maximal possible Gelfand-Kirillov dimension in $\mathcal{O}^{\mathfrak{p}_I}$ (we will omit $I$ if there is no confusion).
That is, $\GK(M_I(\lambda))=\dim (\mathfrak{u}_I)$.

\begin{lemma}Let  $\mathfrak{g}=\mathfrak{so}(2n+1,\mathbb{C}) $ or $\mathfrak{sp}(n,\mathbb{C})$.  Suppose $\mathfrak{p}$ is a standard parabolic subalgebra of type $(n_1, n_2,\cdots,n_k)$.
Then we have 
 $$|\Phi^+_I|=\frac{1}{2}\sum\limits_{j=1 }^{k-1}n_j(n_j-1)+n_k^2.$$

\end{lemma}

\begin{lemma}\label{du}
Let  $\mathfrak{g}=\mathfrak{so}(2n,\mathbb{C}) $. Suppose $\mathfrak{p}$ is a standard parabolic subalgebra of type $(n_1, n_2,\cdots,n_k)$.
Then we have
$$|\Phi^+_I|=\begin{cases}
			\frac{1}{2}\sum\limits_{j=1 }^{k-1}n_j(n_j-1)+n_k^2-n_k &\text{ if }n_k\neq 1,\\
			\frac{1}{2}\sum\limits_{j=1 }^{k-2}n_j(n_j-1)+\frac{1}{2}(n_{k-1}+1)n_{k-1} &\text{ if }n_k=1.
		\end{cases}$$

\end{lemma}

So $\dim (\mathfrak{u})=n^2-|\Phi^+_I|$ for types $B$ and $C$, and $\dim (\mathfrak{u})=n^2-n-|\Phi^+_I|$ for type $D$. We denote this number by $d_m(\mathfrak{p})$. Now our problem is to find out all simple modules $L(\lambda)$ in  $\mathcal{O}^{\mathfrak{p}}$ with maximal Gelfand-Kirillov dimension $d_m(\mathfrak{p}).$

Let $\Gamma$ be a totally ordered set. We denote by $\mathrm{Seq}_n (\Gamma)$ the set of sequences $x=(x_1,x_2,\cdots,x_n)$ with $x_i\in \Gamma$.
 From Bai-Xie \cite{BXie}, we know that for each  $\lambda\in \mathrm{Seq}_n (\Gamma)$, by using Robinson-Schensted insertion algorithm, there is a Young tableau $P(\lambda)$ corresponding to it.
 We recall this method from Bai-Xie \cite{BXie}.
Write  $\lambda=(\lambda_1,\cdots,\lambda_n)$.  Let $ P_0 $ be an empty Young tableau. Assume that we have constructed the Young tableau $ P_k $ associated to $ (\lambda_1,\cdots,\lambda_k) $, $ 0\leq k<n $. Then $ P_{k+1} $ is obtained by adding $ \lambda_{k+1} $ to $ P_k $ as follows. First add $ \lambda_{k+1} $ to the first row of $ T_k $ by replacing the leftmost entry $ x $ in the first row which is strictly bigger than $ \lambda_{k+1} $.  (If there is no such an entry $ x $, we just add a box with entry $\lambda_{k+1}  $ to the right side of the first row, and end this process.) Then add $ x $ to the next row as the same way of adding $ \lambda_{k+1} $ to the first row.  Then we put $P(\lambda)=P_n$.
 The shape of $P(\lambda)$ is denoted by $p(\lambda)$. We identify $p(\lambda)$ with the partition ${\bf p}$ corresponding to the Young tableau $P(\lambda)$.

\begin{definition}\label{fafbfd}
	Let $x\in \mathrm{Seq}_n (\Gamma)$. Define
	\begin{equation*}
		\begin{aligned}
			F_b(x):&=\sum_{i\geq1}(i-1)p_i\od=\sum_{2\nmid i}(q_i\od)^2+\sum_{2\mid i}q_i\od(q_i\od-1),\\
			F_d(x):&=\sum_{i\geq1}(i-1)p_i\ev=\sum_{2\nmid i}q_i\ev(q_i\ev-1)+\sum_{2\mid i}(q_i\ev)^2,
		\end{aligned}
	\end{equation*}
	where $p=p(x)=(p_1,p_2,\cdots) $ and $q=q(x)=p(x)^t=(q_1, q_2, \cdots)$.
\end{definition}

For $ x=(x_1,x_2,\cdots,x_n)\in \mathrm{Seq}_n (\Gamma) $, set
\begin{equation*}
	\begin{aligned}
		{x}^-=&(x_1,x_2,\cdots,x_{n-1}, x_n,-x_n,-x_{n-1},\cdots,-x_2,-x_1),\\
		{}^-{x}=&(-x_n,-x_{n-1},\cdots, -x_2,-x_1,x_1,x_2,\cdots, x_{n-1}, x_n).
	\end{aligned}
\end{equation*}

\begin{proposition}[{\cite[Theorem 1.5]{BXX}}]\label{thm2}
	Let $\lambda=(\lambda_1, \lambda_2, \cdots, \lambda_n)\in \mathfrak{h}^*$ be an integral weight. Then 
	\begin{equation*}
		 \GK L(\lambda)=\begin{cases}
			n^2-F_b(\lambda^-) &\text{ if }\Phi=B_n/C_n,\\
			n^2-n-F_d(\lambda^-) &\text{ if }\Phi=D_n.
		\end{cases}
	\end{equation*}
\end{proposition}

Let  $\mathfrak{g}=\mathfrak{so}(2n+1,\mathbb{C}) $ or $\mathfrak{sp}(n,\mathbb{C})$.  
When $\mathfrak{p}$ is a standard parabolic subalgebra of type  $(n_1,n_2,\cdots,n_k)$, the maximal Gelfand-Kirillov dimension of simple modules in $\mathcal{O}^{\mathfrak{p}}$ will be $\dim(\mathfrak{u})=n^2-(\frac{1}{2}\sum_{j=1}^{k-1}n_j(n_j-1)+n_k^2)=n^2-F_b(\lambda^-)$. 

So a simple module $L(\lambda)$ in $\mathcal{O}^{\mathfrak{p}}$ is socular if and only if
$ F_b(\lambda^-)=|\Phi^+_I|,$
which is
\[ \sum_{i\geq1}(i-1)p_i\od=\frac{1}{2}\sum\limits_{j=1}^{k-1}n_j(n_j-1)+n_k^2.\]

\subsection{H-algorithm}
For an element $w$ in the Weyl group
$W$ of some Lie algebras, we have a partition $p({}^-{w})$.   We call it a {\it  partition of domino type} since it has the same shape as a domino tableau \cite[Proposition 4.6]{BXX}. In general, a partition $\bf p$ (Young diagram $P$) is also called a  partition (Young diagram) of domino type if it has the same corresponding diagram as some partition $p({}^-{w})$ for some $w\in W$.
We recall the H-algorithms defined in \cite{BMW} which can associate a special partition (in the sense of Lusztig, see \cite{Lus84} or \cite{CM93}) to a given domino type partition such that they have the same odd (resp. even) boxes for type $B$ and $C$ (resp. type $D$).
 \begin{definition}[H-algorithm of type $B$]
    If ${\bf p}$ be a partition of domino type (whose Young diagram is $P$) of $2n$, then we can get a special partition  ${\bf p}^s$ of type $B_n$ by the following steps:
    \begin{enumerate}
     \item Construct  the hollow diagram $P\od$ consisting of odd boxes;
	\item Label the rows starting from $1$ but avoid all the consecutive rows
	ending with the shape  $\tytb{O,\none O}$;
\item Keep even labeled rows unchanged and put $\tytb{E}$ on the end of each odd labeled row;
\item Fill the holes. Then if there are only $2n$ boxes in our new Young diagram, we put a box $\tytb{E}$  below the last row and we are done. If there are  $2n+1$ boxes in our new Young diagram, we are done.
\end{enumerate}
We call the above algorithm  {\it H-algorithm of type $B$}.
     
 \end{definition}

\begin{example}
    Let ${\bf p}=[6,4^3,2,2,1,1]$ be a partition of domino type  of $24$. Then we have
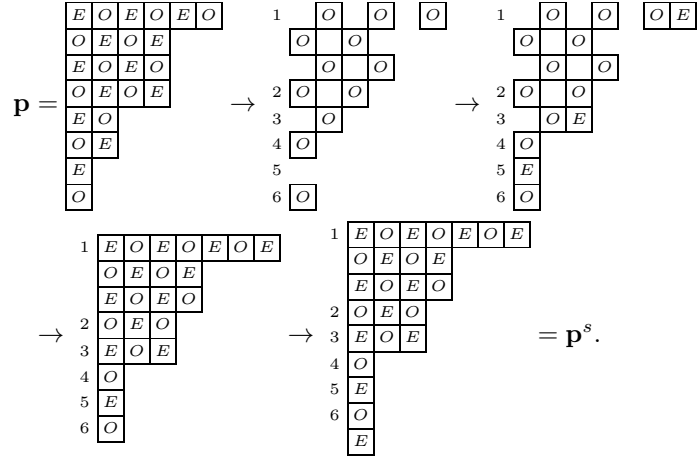
\begin{align*}
    {\bf p}&=\tytb{EOEOEO,OEOE,EOEO,OEOE,EO,OE,E,O} \to
\tytb{{\none[1]}\none O\none O\none O, {\none}O\none O\none,\none\none O\none O, {\none[2]} O\none O,{\none[3]}\none O,{\none[4]} O\none, {\none[5]},{\none[6]}O}
\to
\tytb{{\none[1]}\none O\none O\none O E, {\none}O\none O\none,\none\none O\none O,{\none[2]} O\none O,{\none[3]}\none O E,{\none[4]} O\none,{\none[5]}E, {\none[6]}O}\\
&\to
\tytb{{\none[1]}E OE OE O E ,{\none}OE OE,\none E O E O,{\none[2]} OE O,{\none[3]}E O E,{\none[4]} O\none,{\none[5]}E,{\none[6]}O}
\to \tytb{{\none[1]}E OE OE O E,{\none}OE OE,\none E O E O,{\none[2]} OE O,{\none[3]}E O E,{\none[4]} O\none,{\none[5]}E,{\none[6]}O, \none E}
={\bf p}^s.
\end{align*}

Thus ${\bf p}^s=[7,4,4,3,3,1,1,1,1]$ is a special partition of type $B_{12}$.
\end{example}

\begin{definition}[H-algorithm of type $C$]
    If ${\bf p}$ be a partition of domino type (whose Young diagram is $P$) of $2n$, then we can get a special partition  ${\bf p}^s$ of type $C$ by the following steps:
    \begin{enumerate}
\item Construct  the hollow diagram $P\od$ consisting of odd boxes;
	\item Label the rows starting from $1$ but avoid all the consecutive rows
	ending with the shape  $\tytb{O,\none O}$ (when  two consecutive rows
	has the shape  $\tytb{E,O}$ in $P$, these two rows will not be labeled);
\item Keep odd labeled rows unchanged and put $\tytb{E}$ on the end of each even labeled row;
\item Fill the holes and we are done.
\end{enumerate}
We call the above algorithm {\it Hollow diagram algorithm} or {\it H-algorithm of type $C$}.
\end{definition}

\begin{definition}[H-algorithm of type $D$]
    If ${\bf p}$ be a partition of domino type (whose Young diagram is $P$) of $2n$, then we can get a special partition  ${\bf p}^s$ of type $D_n$ by the following steps:
    \begin{enumerate}
\item Construct  the hollow diagram $P\ev$ consisting of even boxes;
	\item Label the rows starting from $1$ but avoid all the consecutive rows
	ending with the shape  $\tytb{E,\none E}$;
\item Keep odd labeled rows unchanged and put $\tytb{O}$ on the end of each even labeled row;
\item Fill the holes. Then if there are only $2n-1$ boxes in our new Young diagram, we put a box $\tytb{O}$  below the last row and we are done. If there are  $2n$ boxes in our new Young diagram, we are done.
\end{enumerate}
We call the above algorithm  {\it H-algorithm of type $D$}.
\end{definition}

	Given two partitions ${\bf d}=[d_1,\cdots,d_k]$ and ${\bf f}=[f_1,\cdots,f_k]$ of some integer $m$, we say that 	${\bf d}$ {\it dominates } 	${\bf f}$ if the following condition holds:
	\[\label{domi}
	    \sum\limits_{1\leq j\leq l}d_j\geq \sum\limits_{1\leq j\leq l}f_j
	\] for $1\leq l\leq k$.
	\begin{definition}[Collapse]\label{collapse}	
		Let ${\bf d}=[d_1,\cdots,d_k]$ be a partition of $2n+1$. There is a unique largest partition of $2n+1$ of type $B_n$ dominated by ${\bf d}$. If  ${\bf d}$ is not a partition of type $B_n$, then one of its even parts must occur with odd multiplicity. Let $q$ be the largest such part. Then replace the last occurrence of $q$ in ${\bf d}$ by $q-1$ and the first subsequent part $r$ strictly less than $q-1$ by $r+1$. Repeat this process until a partition of type $B_n$ is obtained. This new partition of type $B_n$
		is called the {\it $B$-collapse} of 	${\bf d}$, and we denote it by ${\bf d}_B$. Similarly there are $D$-collapse and $C$-collapse of ${\bf d}$.
	\end{definition}

There is another concept called {\it expansion} which can compute the smallest special partition dominating a given partition ${\bf q}$ of type $X$ (here $X=B,C$ or $D$). This special partition is denoted by ${\bf q}^X$.

	More properties for the  collapse and expansion of partitions can be found in \cite{CM93}.

\begin{remark}
During the process of collapse and expansion, the odd boxes are fixed for types $B$ and $C$ (resp. even boxes are fixed for type $D$) and the moving even box can not meet another even box in the same row (resp. the moving odd box can not meet another odd box in the same row for type $D$). We may call these two operations {\it restricted collapse and expansion}, denoted them by ${\bf q}_{\bar{X}}$ and ${\bf q}^{\bar{X}}$. Then ${\bf q}$, ${\bf q}_{\bar{X}}$ and ${\bf q}^{\bar{X}}$ have the same odd (resp. even) boxes for type $B$ and $C$ (resp. type $D$).
When ${\bf p}$ is a domino type partition, $H({\bf p})$ will be a special partition which has the same odd (resp. even) boxes for type $B$ and $C$ (resp. type $D$). Thus we have $H({\bf p})=({\bf p}_{\bar{X}})^{\bar{X}}$.

\end{remark}

%

\section{Proof of Theorem \ref{mainBC}}\label{proofbc}
%
%

In this section, we will prove Theorem \ref{mainBC}. Let  $\mathfrak{g}=\mathfrak{so}(2n+1,\mathbb{C}) $ or $\mathfrak{sp}(n,\mathbb{C})$ in this section.

Firstly, we
need two lemmas.

\begin{lemma}\label{decreasing}
Suppose $\{m_1,m_2,\cdots,m_k\}$ is a sequence of decreasing positive integers. We have a function $f(x_1,\cdots,x_k)=\sum_{j=1}^{k} x_j^2$, where $(x_1,x_2,\cdots,x_k)\in D$ with 
\begin{align*}D=\{(x_1,\cdots,x_k)\in \mathbb{R}^k\mid &   x_1 \geq m_1, x_1+x_2\geq m_1+m_2,\cdots, 
	\sum_{j=1}^{k-1}x_j\geq \sum_{j=1}^{k-1}m_j,\\
	& \sum_{j=1}^k x_j= \sum_{j=1}^{k}m_j, x_1\geq x_2\geq \cdots\geq x_k\geq 0 \}.
\end{align*}
Then $f$ will take the minimal value
 $\sum_{j=1}^k m_j^2$ if and only if $x_j=m_j$ for all $1\leq j\leq k$.
\end{lemma}
\begin{proof} When  $f$ is a function of two variables, the conclusion is obvious. We assume the conclusion is true for all functions of   $k-1$ variables or less.

Now we assume $f$ is a function of $k$ variables.
We denote $m=\sum_{j=1}^k m_j$. If there is no restriction, the function $f$ will take the minimal value at the point $P_0=(\frac{m}{k},\cdots,\frac{m}{k})$. Now the problem is equivalent to finding out all points on the given plane $\Pi: \sum_{j=1}^k x_j=m$ with some restriction conditions such that the distance from the origin to these points will take the minimal value $d=\sqrt{\sum_{j=1}^k m_j^2}$.
From the condition $x_1 \geq m_1$, $x_1+x_2\geq m_1+m_2, \cdots, \sum_{j=1}^{k-1} x_j\geq \sum_{j=1}^{k-1} m_j$, $x_1\geq x_2\geq \cdots\geq x_k\geq 0$, we know the domain $D\subseteq \Pi$ of our function $f$ is a bounded connected closed subset  in the first quadrant and $P_0\notin D$ (unless all the integers $m_i$ are equal). So $f$ will take its minimal value $d^2$ at the boundary $\partial D$ of $D$.

When $x_1=m_1$, we have $f(m_1,x_2,\cdots,x_k)=m_1^2+\sum_{j=2}^k x_j^2$. We denote $$f_2(x_2,\cdots,x_k)=\sum_{j=2}^k x_j^2,$$ where
 \begin{align*}(x_2,\cdots,x_k)\in D_2&=\{(x_2,\cdots,x_k)\in \mathbb{R}^{k-1}|   x_2 \geq m_2, x_2+x_3\geq m_2+m_3,\cdots,\\
 	& \sum_{j=2}^{k-1} x_j\geq \sum_{j=2}^{k-1} m_j, \sum_{j=2}^{k}x_j= \sum_{j=2}^{k}m_j,  x_2\geq x_3\geq \cdots\geq x_k\geq 0\}.\end{align*}
 So $f_2$ is a function of $k-1$ variables. By our induction, $f_2$ will take its minimal value if and only if   $x_j=m_j$ for all $2\leq j\leq k$. So $f$ will take its minimal value at the boundary $x_1=m_1$ if and only if  $x_j=m_j$ for all $1\leq j\leq k$.

When $x_1+x_2=m_1+m_2$, we have  $$f(x_1,x_2,\cdots,x_k)=(x_1^2+x_2^2)+\sum_{j=3}^{k}x_j^2:=g_2(x_1,x_2)+f_3(x_3,\cdots,x_k).$$
This is a sum of two functions which satisfy our induction. Then  $g_2$ will take its minimal value if and only if $x_1=m_1,x_2=m_2$ and $f_3$ will take its minimal value if and only if   $x_j=m_j$ for all $3\leq j\leq k$. So $f$ will take its minimal value at the boundary $x_1+x_2=m_1+m_2$ if and only if  $x_j=m_j$ for all $1\leq j\leq k$.

We continue this process and finally, when $\sum_{j=1}^k x_j=m=\sum_{j=1}^k m_j$, we will have  \begin{align*}
   f(x_1,x_2,\cdots,x_k)&=x_k^2+\sum_{j=1}^{k-1} x_j^2\\
   &=(m-\sum_{j=1}^{k-1}x_j)^2+\sum_{j=1}^{k-1}x_j^2\\
   &:=f_k(x_1,\cdots,x_{k-1}), 
\end{align*}
where  \begin{align*}(x_1,\cdots,x_{k-1})\in D_k=\{&(x_1,\cdots,x_{k-1})\in \mathbb{R}^{k-1}|   x_1 \geq m_1, x_1+x_2\geq m_1+m_2,\cdots, \\
&\sum_{j=1}^{k-1}x_j\geq \sum_{j=1}^{k-1}m_j,  x_1\geq x_2\geq \cdots\geq x_{k-1}\geq 0\}. \end{align*}

If there is no restriction, the function $f_{k}$ will take the minimal value at the point $Q_0=(\frac{m}{k},\cdots,\frac{m}{k})$. But $Q_0\notin D_k$. So the function $f_{k}$ will take the minimal value at the boundary $\partial D_k$ of $D_k$. When $x_1=m_1, \cdots, \sum_{j=1}^{k-2}x_j= \sum_{j=1}^{k-2}m_j$, the arguments are similar to the case of  $f$.  When  $\sum_{j=1}^{k-1}x_j= \sum_{j=1}^{k-1}m_j$, we have $$f(x_1,x_2,\cdots,x_k)=(\sum_{j=1}^{k-1}x_j)+m_k^2:=g_{k-1}+m_k^2,$$
where $g_{k-1}$ is a function of $k-1$ variables which satisfies our induction.

Thus we have proved our lemma.

\end{proof}

Suppose $\{m_1,m_2,\cdots,m_k\}$ is a sequence of integers. If there exists some index $2\leq s\leq k-1$ such that  $m_{s-1}-1\geq 2m_s\geq m_{s+1} $ and $ m_i\geq m_{i+1} ~ (i\neq s,  s-1)$, the sequence $\{m_1,m_2,\cdots,m_k\}$ will be called an \emph{`almost' descending sequence of index $s$}.

\begin{lemma}\label{decreasingb}
Suppose $\{m_1,m_2,\cdots,m_k\}$ is an `almost' descending sequence of index $s$. Denote $h_t:=\sum\limits_{1\leq j \leq t}m_j $ for $1\leq t\leq k$. We define a function $$f(x_1,\cdots,x_k):=\frac{1}{2}\sum\limits_{1\leq j \leq k,j\neq s}x_j(x_j-1)+x_s^2,$$ where $(x_1,x_2,\cdots,x_k)\in D$ with 
\begin{align*}D=\{(x_1,\cdots,x_k)\in \mathbb{R}^k\mid &   x_1 \geq m_1, x_1+x_2\geq m_1+m_2,\cdots, 	\sum_{j=1}^{k-1}x_j\geq \sum_{j=1}^{k-1}m_j,\\
	& \sum_{j=1}^{k}x_j= \sum_{j=1}^{k}m_j, x_i\geq x_{i+1}(i\neq s,  i+1\neq s), \\
 & x_{s-1}-1\geq 2x_s\geq x_{s+1}\}.
\end{align*}
Then $f$ will take the minimal value
 $\frac{1}{2}\sum\limits_{1\leq j \leq k,j\neq s}m_j(m_j-1)+m_s^2$ if and only if $x_j=m_j$ for all $1\leq j\leq k$.
\end{lemma}
\begin{proof} 
When  $f$ is a function of two variables, the conclusion is obvious. We assume that the conclusion is true for all functions of   $k-1$ variables or less.

Now we assume that $f$ is a function of $k$ variables. Note that $h_k=\sum\limits_{1\leq j \leq k}m_j=\sum\limits_{1\leq j \leq k}x_j$.



If $s\neq k$, we have \begin{align*}
    f(x_1,\cdots,x_k)=&\frac{1}{2}\sum\limits_{1\leq j \leq k,j\neq s}x_j(x_j-1)+x_s^2\\=&\frac{1}{2}\sum_{j=1}^{k-1}x_j^2+\frac{1}{2}x_s^2+\frac{1}{2}x_s+\frac{1}{2}(h_k-\sum_{j=1}^{k-1}x_j)^2-\frac{1}{2}h_k\\:=&g(x_1,\cdots,x_{k-1}).
\end{align*}
Then
\begin{align*}
    g_s:=\frac{\partial g}{\partial x_s}=&2x_s+\frac{1}{2}-(h_k-\sum_{j=1}^{k-1}x_j)=2x_s+\frac{1}{2}-x_k>0, \\
    g_i:=\frac{\partial g}{\partial x_i}=&x_i-(h_k-\sum_{j=1}^{k-1}x_j)=x_i-x_k\geq 0,(i\neq s).
\end{align*}
Thus we have 
\begin{align*}
    f(x_1,\cdots,x_k)=&g(x_1,\cdots,x_{k-1})\\ \geq& g(x_1,\cdots,x_{k-2},h_{k-1}-\sum_{j=1}^{k-2}x_j)\\=&f(x_1,\cdots,x_{k-2},h_{k-1}-\sum_{j=1}^{k-2}x_j,m_k).
\end{align*}
The equality holds if and only if  $x_{k-1}=h_{k-1}-\sum_{j=1}^{k-2}x_j$, which is equivalent to $x_k=m_k$. Then  by induction  $f(x_1,\cdots,x_{k-2},h_{k-1}-\sum_{j=1}^{k-2}x_j,m_k)\geq f(m_1,\cdots,m_k)$ and  the equality holds  if and only if $x_i=m_i$ for all $1\leq i\leq k-2$. Therefore $f(x_1,\cdots,x_k)\geq f(m_1,\cdots,m_k)$ and  the equality holds  if and only if $x_i=m_i$ for all $1\leq i\leq k$.

If $s=k$, similarly by Lemma \ref{decreasing} and the assumption by induction, we can prove that $f(x_1,\cdots,x_k)\geq f(m_1,\cdots,m_k)$, and the
 equality holds if and only if $x_i=m_i$ for all $1\leq i\leq k$.

Thus we have proved our lemma.

\end{proof}

We use  \begin{tabular}{l}
\begin{tikzpicture}[x=0.38cm,y=0.38cm]
\draw[-](0,0) -- (1,0);
\draw[-](1,2) -- (1,0);
\draw[-](0,2) -- (1,2);
\draw[-](0,2) -- (0,0);
\node  at (0.5,1) {\small{A}};
\end{tikzpicture}
\end{tabular}
to denote a vertical rectangle  consisting of two adjacent boxes and 
\begin{tabular}{l}
\begin{tikzpicture}[x=0.38cm,y=0.38cm]
\draw[-](0,0) -- (2,0);
\draw[-](2,1) -- (2,0);
\draw[-](0,1) -- (2,1);
\draw[-](0,1) -- (0,0);
\node  at (1,0.5) {\small{B}};
\end{tikzpicture}
\end{tabular}
to denote a horizontal rectangle  consisting of two adjacent boxes, called {\it  $A$-domino} and {\it  $B$-domino} respectively.

\begin{definition}[Z-diagram]\label{zdigram}
    We construct a Z-diagram of type  $(a_0;b_1,\cdots,b_{k-1})$ by the following steps:
    \begin{enumerate}
    	\item Put $b_i$  $ B$-dominoes in two columns and denote this diagram by  $A_i$ for $1\leq i\leq k-1$ ;
       \item Put $a_0$  $A$-dominoes in one column and denote this diagram by   $A_k$;
\item Rearrange these diagrams such that they are descending by the heights of columns;
\item Put these diagrams together to construct a domino type partition.
\end{enumerate}
The above diagram is called a {\it  Z-diagram of type $(a_0;b_1,\cdots,b_{k-1})$}.
\end{definition}

It is clear that a Z-diagram gives us a domino type partition.

\begin{example}
   The followings are Z-diagrams of type $(4;1,3)$, $(3;3,7,1)$ and $(0;2,7)$ respectively:
$$ \begin{tabular}{l}
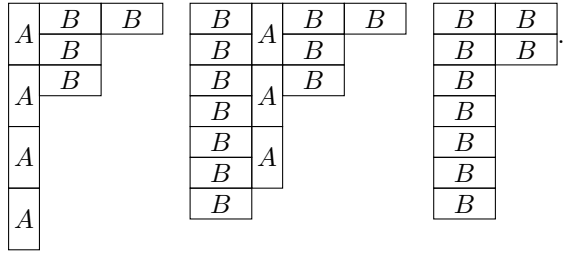
\begin{tikzpicture}[x=0.38cm,y=0.38cm]
\draw[-](0,0) -- (1,0);
\draw[-](0,2) -- (1,2);
\draw[-](0,4) -- (1,4);
\draw[-](0,6) -- (1,6);
\draw[-](0,8) -- (1,8);
\draw[-](1,8) -- (1,0);
\draw[-](0,0) -- (0,8);
\draw[-](1,7) -- (5,7);
\draw[-](1,6) -- (3,6);
\draw[-](1,5) -- (3,5);
\draw[-](1,8) -- (5,8);
\draw[-](3,5) -- (3,8);
\draw[-](5,7) -- (5,8);
\node  at (0.5,1) {\small{A}};
\node  at (0.5,3) {\small{A}};
\node  at (0.5,5) {\small{A}};
\node  at (0.5,7) {\small{A}};

\node  at (2,5.5) {\small{B}};
\node  at (2,6.5) {\small{B}};
\node  at (2,7.5) {\small{B}};
\node  at (4,7.5) {\small{B}};

\end{tikzpicture},
\end{tabular}  
    \begin{tabular}{l}
\begin{tikzpicture}[x=0.38cm,y=0.38cm]
\draw[-](0,0) -- (2,0);
\draw[-](0,1) -- (3,1);
\draw[-](0,2) -- (2,2);
\draw[-](0,3) -- (3,3);

\draw[-](0,4) -- (2,4);
\draw[-](0,0) -- (0,7);
\draw[-](3,4) -- (5,4);
\draw[-](0,5) -- (5,5);
\draw[-](0,6) -- (2,6);
\draw[-](3,6) -- (7,6);
\draw[-](0,7) -- (7,7);
\draw[-](2,0) -- (2,7);

\draw[-](3,1) -- (3,7);
\draw[-](5,4) -- (5,7);
\draw[-](7,6) -- (7,7);
\node  at (1,0.5) {\small{B}};
\node  at (1,1.5) {\small{B}};
\node  at (1,2.5) {\small{B}};
\node  at (1,3.5) {\small{B}};
\node  at (1,4.5) {\small{B}};
\node  at (1,5.5) {\small{B}};
\node  at (1,6.5) {\small{B}};

\node  at (2.5,2) {\small{A}};
\node  at (2.5,4) {\small{A}};
\node  at (2.5,6) {\small{A}};
\node  at (4,4.5) {\small{B}};
\node  at (4,5.5) {\small{B}};
\node  at (4,6.5) {\small{B}};
\node  at (6,6.5) {\small{B}};
\end{tikzpicture},
\end{tabular}  
    \begin{tabular}{l}
\begin{tikzpicture}[x=0.38cm,y=0.38cm]
\draw[-](0,0) -- (2,0);
\draw[-](0,1) -- (2,1);
\draw[-](0,2) -- (2,2);
\draw[-](0,3) -- (2,3);

\draw[-](0,4) -- (2,4);
\draw[-](0,0) -- (0,7);

\draw[-](0,5) -- (4,5);
\draw[-](0,6) -- (4,6);

\draw[-](0,7) -- (4,7);
\draw[-](2,0) -- (2,7);

\draw[-](4,5) -- (4,7);

\node  at (1,0.5) {\small{B}};
\node  at (1,1.5) {\small{B}};
\node  at (1,2.5) {\small{B}};
\node  at (1,3.5) {\small{B}};
\node  at (1,4.5) {\small{B}};
\node  at (1,5.5) {\small{B}};
\node  at (1,6.5) {\small{B}};

\node  at (3,6.5) {\small{B}};
\node  at (3,5.5) {\small{B}};

\end{tikzpicture}.
\end{tabular}  $$
   
\end{example}

    In a Z-diagram, it is obvious that the subdiagram  consisting of $A$-dominoes  always lies in an odd column.

\begin{lemma}\label{countfbfd}
 Suppose   $P$ is a Z-diagram type of $(a_0;b_1,\cdots,b_{k-1})$.Then we have 
    
	\begin{equation*}
		\begin{aligned}
			F_b(P)&=a_0^2+\frac{1}{2}\sum_{i}b_i(b_i-1),\\
			F_d(P)&=a_0^2-a_0+\frac{1}{2}\sum_{i}b_i(b_i-1).
		\end{aligned}
	\end{equation*}
\end{lemma}
\begin{proof}
    Recall in Definition \ref{fafbfd}, we have
	\begin{equation*}
		\begin{aligned}
			F_b(x):&=\sum_{i\geq1}(i-1)p_i\od=\sum_{2\nmid i}(q_i\od)^2+\sum_{2\mid i}q_i\od(q_i\od-1),\\
			F_d(x):&=\sum_{i\geq1}(i-1)p_i\ev=\sum_{2\nmid i}q_i\ev(q_i\ev-1)+\sum_{2\mid i}(q_i\ev)^2,
		\end{aligned}
	\end{equation*}
where $p=p(x)=(p_1,p_2,\cdots, p_n) $ is the shape of the Young tableau  $P(x)$ obtained by using R-S algorithm.
Thus $F_b(x)$ and $F_d(x)$ only depend on the shape of the Young tableau $P(x)$. So we can similarly define $F_b(T)$ for a Young tableau $T$.

 
    Now  $P$ is a Z-diagram of  type  $(a_0;b_1,\cdots,b_{k-1})$. We can regard it as a Young diagram of shape $p=(p_1,..,p_k)$ where $p^t=(a_0',b_1',b_1',b_2',b_2',..,b_{k-1}',b_{k-1}')$ and $(a_0',b_1',b_2',..,b_{k-1}')$ is the  arrangement of the sequence $(a_0,b_1,b_2,..,b_{k-1})$ in descending order. We use $F_b(P|_Q)$ to denote  the value of the subdiagram  $Q$ in $F_b(P)$. Similarly denote $F_d(P|_Q)$.

    Note that the subdiagram $A_k$ has $a_0$ odd boxes, $a_0$ even boxes, and  always lies in an odd column. Thus,  $F_b(P|_{A_k})=a_0^2$ and $F_d(P|_{A_k})=a_0^2-a_0$.
      
      The subdiagram $A_i~ (i\neq k)$ has $b_i$ odd boxes and $b_i$ even boxes, and  $A_i$  always lies in an odd column and an even column. Thus we have $F_b(P|_{A_i})=\frac{1}{2}b_i(b_i-1)=F_d(P|_{A_i})$.

       Sum up these results,  we have proved the result.
\end{proof}

\begin{corollary}\label{numberev}
  Suppose  $p=(p_1,\cdots,p_N)$ is the shape of the Z-diagram  $P$ of type  $(n_k; n_1,\cdots,n_{k-1})$ so that $${p}^t=(m_1,m_1,m_2,m_2,\cdots,m_{s-1},m_{s-1,},2m_s,m_{s+1},m_{s+1},\cdots,m_{k},m_{k}).$$ Here $(m_1,\cdots,m_k)$ is the  arrangement of the sequence $(n_1,n_2,..,n_{k})$ in descending order and $m_s=n_k$. 
  Then we have $$p\od=(\left\lfloor \frac{m_1}{2} \right\rfloor, \left\lceil \frac{m_1}{2}\right\rceil,\cdots,\left\lfloor \frac{m_{s-1}}{2} \right\rfloor, \left\lceil \frac{m_{s-1}}{2}\right\rceil,m_s,\left\lceil \frac{m_{s+1}}{2}\right\rceil,\left\lfloor \frac{m_{s+1}}{2} \right\rfloor,\left\lceil \frac{m_{k}}{2}\right\rceil,\left\lfloor \frac{m_{k}}{2} \right\rfloor)$$ and
  $$p\ev=(\left\lceil \frac{m_1}{2}\right\rceil,\left\lfloor \frac{m_1}{2} \right\rfloor, \cdots,\left\lceil \frac{m_{s-1}}{2}\right\rceil,\left\lfloor \frac{m_{s-1}}{2} \right\rfloor,  m_s,\left\lfloor \frac{m_{s+1}}{2} \right\rfloor,\left\lceil \frac{m_{s+1}}{2}\right\rceil,\left\lfloor \frac{m_{k}}{2} \right\rfloor,\left\lceil \frac{m_{k}}{2}\right\rceil).$$
\end{corollary}

\begin{lemma}
 Suppose   $P_1$ is a Z-diagram  of  type $(a_0; b_1,\cdots,b_{k-1})$ and    $P_2$ is a Z-diagram of type $(a_0; c_1,\cdots,c_{k-1})$.
Then we have $P_1=P_2$ if and only if there exists some $ \sigma \in S_{k-1}$ such that $c_i=\sigma(b_i)$ for $1\leq i\leq k-1$.
    
\end{lemma}

\begin{proof}
   The case for  $k=2$ is trivial. From the construction of Z-diagram, $(c_1,\cdots,c_{k-1})$ is an arrangement of $(b_1,\cdots,b_{k-1})$, thus it is obvious to prove the result.
\end{proof}

Suppose    $\Phi=B_n/C_n/D_n$ and $\lambda$ is an integral weight. Then from \cite{BMW}, we know that  the shape $p=p(\lambda^-)$ of the Young tableau $P(\lambda^-)$ gives a partition of domino type. By using H-algorithm, we can get a special partition from this Young tableau $P(\lambda^-)$.


\begin{lemma}\label{min}
    Suppose  $A$ is a Young diagram of domino type with two columns and  $c_i(A)$ (or $c_i$) is the number of boxes in the $i$-th column of $A$. Then $F_b(A)$ is a decreasing function of $c_2(A)$ when $c_1(A)+c_2(A)=2m$ is a fixed number. It will take the minimal value $\frac{1}{2}m(m-1)$ if and only if $c_2(A)=c_1(A)$ or $c_2(A)=c_1(A)-2$.
 \end{lemma}
\begin{proof}
  Recall that $F_b(x)=\sum_{2\nmid i}(q_i\od)^2+\sum_{2\mid i}q_i\od(q_i\od-1)$. Since $A$ is a Young diagram  with two columns, we  may write  $$F_b(A)=(c_1\od)^2+c_2\od(c_2\od-1):=F(c_1, c_2)=F(n-c_2,c_2)=f(c_2).$$ When we move one box from the first column to the second column of $A$, we will get a new diagram $\bar{A}$. Thus $\bar{c}_1=c_1-1$ and $\bar{c}_2=c_2+1$. Since $A$ is a Young diagram of domino type, there are two possibilities as follows:
  \begin{itemize}
      \item $\bar{c}_1\od=c_1\od$, $\bar{c}_2\od=c_2\od$;
     \item $\bar{c}_1\od=c_1\od-1$, $\bar{c}_2\od=c_2\od+1$.
  \end{itemize}
In the second case, we have $c_1\od\geq c_2\od+1$. Thus we have \begin{align*}
    f(\bar{c}_2)&=(\bar{c}_1\od)^2+\bar{c}_2\od(\bar{c}_2\od-1)\\
    &=(c_1\od-1)^2+({c}_2\od+1)c_2\od\\
    &=(c_1\od)^2+c_2\od(c_2\od-1)+(-2c_1\od+1+2c_2\od)\\
    &\leq f(c_2)+(-2c_2\od-2+1+2c_2\od)\\
    &=f(c_2)-1\\
    &< f(c_2).
\end{align*} 

This finishes the proof.

\end{proof}

Suppose $\mathfrak{g}=\mathfrak{sl}(n,\mathbb{C})$. Let $\mathfrak{p}$ be a standard parabolic subalgebra of type ${(n_1, n_2,\cdots,}$ ${n_k)}$. Suppose $L(\lambda)\in \mathcal{O}^{\mathfrak{p}}$ is an integral  highest weight module. Write $c_i$ for the number of  entries in the $ i $-th column of $P(\lambda)$.
We say  $\lambda$ is a {\it maximal standard weight of type $(n_1, n_2,\cdots,n_k)$}  if  the sequence $(c_1,..,c_k)$ is the arrangement of the sequence $(n_1, n_2,\cdots,n_k)$ in descending order. When  $\mathfrak{g}=\mathfrak{so}(2n+1,\mathbb{C}) $ or  $\mathfrak{sp}(n,\mathbb{C})$, we can consider $\lambda^-$ as a weight corresponding to a   standard parabolic subalgebra of  type $(n_1, n_2,\cdots,n_{k-1},2n_k,n_{k-1},\cdots,n_1)$
(if $n_k>0$) or type $(n_1, n_2,\cdots,n_{k-1},0,0,n_{k-1},\cdots,n_1)$ (if $n_k=0$) in $\mathfrak{g}=\mathfrak{sl}(2n,\mathbb{C})$. Similarly we say $\lambda$ is a {\it maximal standard weight of type $(n_1, n_2,\cdots,n_k)$} if $\lambda^-$ is a maximal standard weight of  type $(n_1, n_2,\cdots,n_{k-1},2n_k,n_{k-1},\cdots,n_1)$ (if $n_k>0$)
 or type $(n_1, n_2,\cdots,n_{k-1},0,0,   n_{k-1},\cdots,n_1)$ (if $n_k=0$).

\begin{lemma}[Uniqueness of Types $B$ and $C$]\label{UniquenessBC}
Suppose  $\mathfrak{g}=\mathfrak{so}(2n+1,\mathbb{C}) $ or $\mathfrak{sp}(n,\mathbb{C})$. Let $\frp$ be a  standard parabolic subalgebra of type $(n_1, n_2,\cdots,n_k)$.
Suppose $L(\lambda)\in \mathcal{O}^{\mathfrak{p}}$ is an integral  highest weight module.  The number $F_b(\lambda^-)$ will take the minimal value
$\frac{1}{2}\sum_{j=1}^{k-1}n_j(n_j-1)+n_k^2$ if and only if $\lambda$ is a  maximal standard weight of type $(n_1, n_2,\cdots,n_k)$. In other words, $P(\lambda^-)\od$ is unique when $F_b(\lambda^-)$ takes the minimal value.
\end{lemma}

\begin{proof}
 Since  $L(\lambda)\in \mathcal{O}^{\mathfrak{p}}$, in $\lambda^-$, for each $n_i$ we have two decreasing subsequences $$(\lambda_{n_{i-1}+1}, \lambda_{n_{i-1}+2},\cdots,\lambda_{n_{i-1}+n_i})$$  which is denoted by $str_i^+$, and $$(-\lambda_{n_{i-1}+n_i}, -\lambda_{n_{i-1}+n_i-1},\cdots,-\lambda_{n_{i-1}+1})$$ which is denoted by $str_i^-$.
 
    We sort $(n_1, n_2,\cdots,n_k)$  `almost' descending to get a sequence $(m_1, m_2,\cdots,m_k)$ where $m_s=n_k$, $m_{s-1}-1\geq 2m_s\geq m_{s+1} $ and $ m_i\geq m_{i+1}$ (if $i\neq s, s-1$). Obviously, each $m_i$ is equal to some $n_j$. We call it $n_{j_i}$. And it is corresponding to two decreasing subsequences $str_{j_i}^+$ and $str_{j_i}^-$. If $n_{j_s}=n_k> 0$, the subsequence $(str_{j_s}^+, str_{j_s}^-)$ will be a decreasing subsequence.
    
    Now we construct $P(\lambda^-)$, and we only focus on odd boxes.
    
    {\bf Step 1.} We construct a Young tableau $Y_1$ corresponding to $m_1=n_{j_1}$. Actually we choose the subsequence $(str_{j_1}^+, str_{j_1}^-)$ to construct the Young tableau $Y_1$.
   Then we find that there are $m_1$ odd boxes in $Y_1$. If $s=1$, the Young tableau $Y_1$ contains only  one column. If $s\neq 1$, the Young tableau $Y_1$ has at most  two columns.
    There are $m_1$ odd boxes in $Y_1$. We denote the columns corresponding to   $Y_1$ by Area $S_1$.

    {\bf Step 2.} We continue to construct the Young tableau with boxes corresponding to $m_2=n_{j_2}$. We insert $str_{j_2}^+$ and $str_{j_2}^-$ into the subsequence $(str_{j_1}^+, str_{j_1}^-)$ according to their original order in $\lambda^-$, then  the new subsequence  will  generate a new Young tableau $Y_2$. And we find that there are $m_1+m_2$ odd boxes in $Y_2$. Comparing to $Y_1$, we get some new columns in $Y_2$, and we denote them by Area $S_2$. With some new elements coming in, there may be some more odd boxes in the new Area $S_1$.

    {\bf Step 3.} Similarly, we continue these steps to $m_k$. Then We have constructed $P(\lambda^-)$.

    We use $x_i$ to denote the number of odd boxes in the   Area $S_i$ of  $P(\lambda^-)$. Note that the Area $S_s$ contains only one column when it is not empty, and the other Areas contain two columns. By the steps of  constructing $P(\lambda^-)$, we find that
   \begin{align*}			
        & x_1 \geq m_1, x_1+x_2\geq m_1+m_2,\cdots, 
        	\sum_{j=1}^{k-1}x_j\geq \sum_{j=1}^{k-1}m_j,\\
	&\sum_{j=1}^{k}x_j= \sum_{j=1}^{k}m_j, x_i\geq x_{i+1}(i\neq s, s-1), 
  x_{s-1}-1\geq 2x_s\geq x_{s+1}.
		\end{align*}
	
By Lemma \ref{min}, the Area $S_i$ ($i\neq s$) will contain two columns if we want the value of $F_b(\lambda^-)$ to be as small as possible.

    Then by Lemma \ref{decreasingb} and Lemma \ref{min}, $F_b(\lambda^-)$ will take the minimal value
$$\frac{1}{2}\sum_{j=1}^{k-1}n_j(n_j-1)+n_k^2$$ if and only if $x_i=m_i$ for all $1\leq i\leq k$.
Thus, the numbers of  odd boxes in different areas are fixed. 

This finishes the proof.

\end{proof}


    

    In Lemma \ref{UniquenessBC}, we constructed  a Young tableau $P(\lambda^-)$ whose odd Young tableau $P(\lambda^-)\od$ is unique when $F_b(\lambda^-)$  takes the minimal value. In fact we can construct a Z-diagram which has same odd boxes with $P(\lambda^-)\od$.
\begin{theorem}\label{Z-diagram of odd boxes}
  Suppose  $\mathfrak{g}=\mathfrak{so}(2n+1,\mathbb{C}) $ or $\mathfrak{sp}(n,\mathbb{C})$. Let $\frp$ be a  standard parabolic subalgebra of type $(n_1, n_2,\cdots,n_k)$.
Suppose $L(\lambda)\in \mathcal{O}^{\mathfrak{p}}$ is an integral  highest weight module.  The Z-diagram $ P$ of type $(n_k; n_1, n_2,\cdots, n_{k-1})$ has the same odd boxes as the odd Young tableau  $P(\lambda^-)\od$ when $F_b(\lambda^-)$ takes the minimal value.
\end{theorem}
\begin{proof}
    From Lemma \ref{countfbfd}, we know $F_b(P)=F_b(\lambda^-)$. Thus the result follows from Lemma \ref{UniquenessBC} since $P(\lambda^-)\od$ is unique.
\end{proof}

Now we have proved  the main Theorem \ref{mainBC} for types $B$ and $C$.

\begin{example}
 Let   $\mathfrak{g}=\mathfrak{so}(9,\mathbb{C})$.  Suppose $\Delta\setminus I=\{\alpha_{2},\alpha_{3}\}$. Then the corresponding   parabolic subalgebra 
  $\mathfrak{p}_I$ is a standard parabolic subalgebra of type  $(2,1,1)$. By Theorem \ref{mainBC}, a simple integral highest weight  module $L(\lambda)$  in $\mathcal{O}^{\mathfrak{p}}$  is socular  if and only if $P(\lambda^-)$ has the same odd boxes as the following Z-diagram $P$ of type  $(1;2,1)$:  
\begin{center}
 \begin{tabular}{l}
\begin{tikzpicture}[x=0.38cm,y=0.38cm]
\draw[-](0,0) -- (3,0);
\draw[-](0,1) -- (2,1);
\draw[-](0,2) -- (5,2);
\draw[-](2,0) -- (2,2);
\draw[-](0,0) -- (0,2);
\draw[-](3,1) -- (5,1);
\draw[-](3,0) -- (3,2);
\draw[-](5,1) -- (5,2);

\node  at (1,0.5) {\small{B}};
\node  at (1,1.5) {\small{B}};

\node  at (2.5,1) {\small{A}};

\node  at (4,1.5) {\small{B}};

\end{tikzpicture}.
\end{tabular}




\end{center}

The shape of $P(\lambda^-)\od$ is $p(\lambda^-)\od=p\od=(2,2)$, where $p$ is the shape of the Z-diagram $P$.

Now  suppose $\lambda=(-5,-6,-4,2)\in \mathfrak{h}^*$. Then $L(\lambda)$  in $\mathcal{O}^{\mathfrak{p}}$  is socular. We can check that $P(\lambda^-)$ has the same odd boxes as the  Z-diagram $P$ of type  $(1; 2, 1)$. 
Actually,  since $\lambda^-=(-5,-6,-4,2, -2,4,6,5)$ and from the R-S algorithm, we have

\begin{align*}
 &{\begin{tikzpicture}[scale=\domscale+0.1,baseline=-20pt]
			\hobox{0}{0}{-6}
			\hobox{0}{1}{-5}
			
	\end{tikzpicture}}\to
	{\begin{tikzpicture}[scale=\domscale+0.1,baseline=-20pt]
			\hobox{0}{0}{-6}
			\hobox{0}{1}{-5}
            \hobox{1}{0}{-4}
            \hobox{2}{0}{2}
           
	\end{tikzpicture}}\to
	{\begin{tikzpicture}[scale=\domscale+0.1,baseline=-20pt]
			\hobox{0}{0}{-6}
			\hobox{0}{1}{-5}
            \hobox{1}{0}{-4}
            \hobox{1}{1}{2}
            \hobox{2}{0}{-2}
           
	\end{tikzpicture}}\to
	{\begin{tikzpicture}[scale=\domscale+0.1,baseline=-20pt]
		\hobox{0}{0}{-6}
			\hobox{0}{1}{-5}
            \hobox{1}{0}{-4}
            \hobox{1}{1}{2}
            \hobox{2}{0}{-2}
             \hobox{3}{0}{4}
              \hobox{4}{0}{6}
	\end{tikzpicture}}\\
 &\to
	{\begin{tikzpicture}[scale=\domscale+0.1,baseline=-20pt]
		\hobox{0}{0}{-6}
			\hobox{0}{1}{-5}
            \hobox{1}{0}{-4}
            \hobox{1}{1}{2}
            \hobox{2}{1}{6}
            \hobox{2}{0}{-2}
             \hobox{3}{0}{4}
              \hobox{4}{0}{5}
	\end{tikzpicture}}
  =P(\lambda^-).
	\end{align*}

\end{example}

\section{Proof of Theorem \ref{mainD}}\label{proofd}
In this section, let $\mathfrak{g}=\mathfrak{so}(2n,\mathbb{C})$.
Suppose $\mathfrak{p}$ is a standard parabolic subalgebra of type $(n_1, n_2,\cdots,n_k)$ ($n_k\neq 1$).
Based on Proposition \ref{thm2}, the maximal Gelfand-Kirillov dimension of simple modules in $\mathcal{O}^{\mathfrak{p}}$ will be $$d_m(\mathfrak{p})=n^2-n-(\frac{1}{2}\sum_{j=1}^{k-1}n_j(n_j-1)+n_k^2-n_k)=n^2-n-F_d(\lambda^-).$$

So a simple module $L(\lambda)$ in $\mathcal{O}^{\mathfrak{p}}$ is socular if and only if
\[ F_d(\lambda^-)=\frac{1}{2}\sum_{j=1}^{k-1}n_j(n_j-1)+n_k^2-n_k,\]
which is
\[ \sum_{i\geq1}(i-1)p_i\ev=\frac{1}{2}\sum_{j=1}^{k-1}n_j(n_j-1)+n_k^2-n_k.\]

\begin{lemma}\label{decreasingd}
Suppose $\{m_1,m_2,\cdots,m_k\}$ is a sequence of integers with $m_{s-1}\geq 2m_s\geq m_{s+1}+1 $ and $ m_i\geq m_{i+1}(i\neq s,  s-1) $. We say $\{m_1,m_2,\cdots,m_k\}$ is sorted `almost' descending. Denote $h_t:=\sum\limits_{1\leq j \leq t}m_j $. We define a function $$f(x_1,\cdots,x_k):=\frac{1}{2}\sum\limits_{1\leq j \leq k, j\neq s}x_j(x_j-1)+x_s^2-x^s,$$ where $(x_1,x_2,\cdots,x_k)\in D$ with 
\begin{align*}D=\{(x_1,\cdots,x_k)\in \mathbb{R}^k\mid &   x_1 \geq m_1, x_1+x_2\geq m_1+m_2,\cdots, 	\sum_{j=1}^{k-1}x_j\geq \sum_{j=1}^{k-1}m_j,\\
	& \sum_{j=1}^{k}x_j= \sum_{j=1}^{k}m_j, x_i\geq x_{i+1}(i\neq s,  s-1), \\
 & x_{s-1}\geq 2x_s\geq x_{s+1}+1\}.
\end{align*}
Then $f$ will take the minimal value
 $\frac{1}{2}\sum\limits_{1\leq j \leq k,j\neq s}m_j(m_j-1)+m_s^2-m_s$ if and only if $x_j=m_j$ for all $1\leq j\leq k$.
\end{lemma}
\begin{proof}
    Similar to the proof in Lemma \ref{decreasingb}.
\end{proof}


Since  $\mathfrak{g}=\mathfrak{so}(2n,\mathbb{C})$, we can consider $\lambda^-$ as a weight corresponding to a   standard parabolic subalgebra of  type $(n_1, n_2,\cdots,n_{k-1},2n_k,n_{k-1},\cdots,n_1)$
(if $n_k>0$ and $\lambda_n>0$), or type $(n_1, n_2,\cdots,n_{k-1},2n_k-1,1,n_{k-1},\cdots,n_1)$ (if $n_k>0$ and $\lambda_n \leq 0$), or type $(n_1, n_2,\cdots,n_{k-1},0,0,n_{k-1},\cdots,n_1)$ (if $n_k=0$) in $\mathfrak{g}=\mathfrak{sl}(2n,\mathbb{C})$. Similarly we say $\lambda$ is a {\it maximal standard weight of type $(n_1, n_2,\cdots,n_k)$} if $\lambda^-$ is a maximal standard weight of   type $(n_1, n_2,\cdots,n_{k-1},2n_k,n_{k-1},\cdots,n_1)$
(if $n_k>0$ and $\lambda_n>0$), or type $(n_1, n_2,\cdots,n_{k-1},2n_k-1,1,n_{k-1},\cdots,n_1)$ (if $n_k>0$ and $\lambda_n \leq 0$), or type $(n_1, n_2,\cdots,n_{k-1},0,0,n_{k-1},\cdots,n_1)$ (if $n_k=0$) in $\mathfrak{g}=\mathfrak{sl}(2n,\mathbb{C})$.

\begin{lemma}[Uniqueness of Type $D$]\label{UniquenessD}
Suppose  $\mathfrak{g}=\mathfrak{so}(2n,\mathbb{C})$. Let $\frp$ be a  standard parabolic subalgebra of type $(n_1, n_2,\cdots,n_k)$ ($n_k\neq 1$).
Suppose $L(\lambda)\in \mathcal{O}^{\mathfrak{p}}$ is an integral  highest weight module. The number $F_d(\lambda^-)$ will take the minimal value
$\frac{1}{2}\sum_{j=1}^{k-1}n_j(n_j-1)+n_k^2-n_k$ if and only if $\lambda$ is a  maximal standard weight of type $(n_1, n_2,\cdots,n_k)$. In other words, $P(\lambda^-)\ev$ is unique when $F_d(\lambda^-)$ takes the minimal value.
\end{lemma}

\begin{proof}
The proof of this case  is similar to  that in types $B$ and $C$. However, we have to focus on even boxes  rather than odd boxes.

     We sort $(n_1, n_2,\cdots,n_k)$  `almost' descending to get a sequence $(m_1, m_2,\cdots,m_k)$ where $m_s=n_k$,  $m_{s-1}\geq 2m_s\geq m_{s+1}+1 $ and $ m_i\geq m_{i+1}(i\neq s,  s-1)$. Obviously, $m_i$ is equal to some $n_j$. We call it $n_{j_i}$. And it is corresponding to two decreasing subsequences $str_{j_i}^+$ and $str_{j_i}^-$.
    
    Now we construct $P(\lambda^-)$, and we only focus on even boxes.
    
    {\bf Step 1.} We construct a Young tableau $Y_1$ corresponding to $m_1=n_{j_1}$. Actually we choose the subsequence $(str_{j_1}^+, str_{j_1}^-)$ to construct the Young tableau $Y_1$.
   Then we find that there are $m_1$ even boxes in $Y_1$. If $s=1$, the Young tableau $Y_1$ contains only  one column or two columns with $c_2(Y_1)=1$. We denote the first column in $Y_1$  by Area $S_1$. If $s\neq 1$, the Young tableau $Y_1$ has at most  two columns.
    There are $m_1$ even boxes in $Y_1$. We denote the columns corresponding to   $Y_1$ by Area $S_1$.

    {\bf Step 2.} We continue to construct the Young tableau with boxes corresponding to $m_2=n_{j_2}$. We insert $str_{j_2}^+$ and $str_{j_2}^-$ into the subsequence $(str_{j_1}^+, str_{j_1}^-)$ according to their original order in $\lambda^-$, then  the new subsequence  will  generate a new Young tableau $Y_2$. And we find that there are $m_1+m_2$ even boxes in $Y_2$. Comparing to Area $S_1$, we get some other columns in $Y_2$ (it may happen that $c_4(Y_2)=1$). We use Area $S_2$ to denote the new columns if $c_4(Y_2)=0$ or denote the first two new columns if $c_4(Y_2)>0$. With some new elements coming in, there may be some more even boxes in the new Area $S_1$.

     {\bf Step 3.} Similarly, we continue these steps to $m_k$. Then We have constructed $P(\lambda^-)$.

 We use $x_i$ to denote the number of even boxes in the   Area $S_i$ of  $P(\lambda^-)$. Note that the Area $S_s$ contains only one column containing even boxes when it is not empty, and the other Areas contain two columns. By the steps of  constructing $P(\lambda^-)$, we find that
    \begin{align*}			
        & x_1 \geq m_1, x_1+x_2\geq m_1+m_2,\cdots, 
        	\sum_{j=1}^{k-1}x_j\geq \sum_{j=1}^{k-1}m_j,\\
	&\sum_{j=1}^{k}x_j= \sum_{j=1}^{k}m_j, x_i\geq x_{i+1}(i\neq s, s-1), 
  x_{s-1}\geq 2x_s\geq x_{s+1}+1.
		\end{align*}

 By Lemma \ref{min}, the Area $S_i$ ($i\neq s$) will contain two columns if we want the value of $F_d(\lambda^-)$ to be as small as possible.

    Then by Lemma \ref{decreasingd} and Lemma \ref{min}, $F_d(\lambda^-)$ will take the minimal value
$$\frac{1}{2}\sum_{j=1}^{k-1}n_j(n_j-1)+n_k^2-n_k$$ if and only if $x_i=m_i$ for all $1\leq i\leq k$.
Thus, the numbers of  even boxes in different areas are fixed. 

This finishes the proof.
\end{proof}
 In Lemma \ref{UniquenessD}, we constructed  a Young tableau $P(\lambda^-)$ whose even Young tableau $P(\lambda^-)\ev$ is unique when $F_d(\lambda^-)$  takes the minimal value. In fact we can construct a Z-diagram which has same even boxes with $P(\lambda^-)\ev$.
\begin{theorem}\label{Z-diagram of even boxes}
  Suppose  $\mathfrak{g}=\mathfrak{so}(2n,\mathbb{C})$. Let $\frp$ be a  standard parabolic subalgebra of type $(n_1, n_2,\cdots,n_k)$ ($n_k\neq 1$).
Suppose $L(\lambda)\in \mathcal{O}^{\mathfrak{p}}$ is an integral  highest weight module.  The Z-diagram $ P$ of type $(n_k; n_1, n_2,\cdots, n_{k-1})$ has the same even boxes as the even Young tableau  $P(\lambda^-)\od$ when $F_d(\lambda^-)$ takes the minimal value.
\end{theorem}
\begin{proof}
    From Lemma \ref{countfbfd}, we know $F_d(P)=F_d(\lambda^-)$. Thus the result follows from Lemma \ref{UniquenessD} since $P(\lambda^-)\ev$ is unique.
\end{proof}

Recall that the case of $n_k=1$ can be reduced to the case of $n_k=0$, so it is unnecessary to talk about it here. 
Thus we have proved  the main Theorem \ref{mainD} for type $D$.

\begin{example}
 Let   $\mathfrak{g}=\mathfrak{so}(10,\mathbb{C})$.  Suppose $\Delta\setminus I=\{\alpha_{1},\alpha_{3},\alpha_{5}\}$. Then the corresponding   parabolic subalgebra 
  $\mathfrak{p}_I$ is a standard parabolic subalgebra of type  $(1,2,2,0)$. By Theorem \ref{mainD}, a simple integral highest weight  module $L(\lambda)$  in $\mathcal{O}^{\mathfrak{p}}$  is socular  if and only if $P(\lambda^-)$ has the same even boxes as the following Z-diagram $P$ of type  $(0;1,2,2)$:  
\begin{center}


 \begin{tabular}{l}
\begin{tikzpicture}[x=0.38cm,y=0.38cm]
\draw[-](0,0) -- (4,0);
\draw[-](0,1) -- (6,1);
\draw[-](0,2) -- (6,2);
\draw[-](2,0) -- (2,2);
\draw[-](0,0) -- (0,2);

\draw[-](4,0) -- (4,2);
\draw[-](6,1) -- (6,2);

\node  at (1,0.5) {\small{B}};
\node  at (1,1.5) {\small{B}};

\node  at (3,0.5) {\small{B}};

\node  at (3,1.5) {\small{B}};
\node  at (5,1.5) {\small{B}};

\end{tikzpicture}.
\end{tabular}

\end{center}

The shape of $P(\lambda^-)\ev$ is $p(\lambda^-)\ev=p\ev=(3,2)$, where $p$ is the shape of the Z-diagram $P$.

Now  suppose $\lambda=(-6,-4,-5,-2,-3)\in \mathfrak{h}^*$. Then $L(\lambda)$  in $\mathcal{O}^{\mathfrak{p}}$  is socular. We can check that $P(\lambda^-)$ has the same even boxes as the  Z-diagram $P$ of type  $(0; 1, 2, 2)$. 
Actually,  since $\lambda^-=(-6,-4,-5,-2,-3,3,2,5,4,6)$ and from the R-S algorithm, we have

\begin{align*}
 &{\begin{tikzpicture}[scale=\domscale+0.1,baseline=-13pt]
			\hobox{0}{0}{-6}
			\hobox{1}{0}{-4}
			
	\end{tikzpicture}}\to
	{\begin{tikzpicture}[scale=\domscale+0.1,baseline=-20pt]
			\hobox{0}{0}{-6}
			\hobox{0}{1}{-4}
            \hobox{1}{0}{-5}
            \hobox{2}{0}{-2}
           
	\end{tikzpicture}}\to
	{\begin{tikzpicture}[scale=\domscale+0.1,baseline=-20pt]
		\hobox{0}{0}{-6}
			\hobox{0}{1}{-4}
   \hobox{1}{1}{-2}
            \hobox{1}{0}{-5}
            \hobox{2}{0}{-3}

	\end{tikzpicture}}\to
	{\begin{tikzpicture}[scale=\domscale+0.1,baseline=-20pt]
		\hobox{0}{0}{-6}
			\hobox{0}{1}{-4}
   \hobox{1}{1}{-2}
            \hobox{1}{0}{-5}
            \hobox{2}{0}{-3}
            \hobox{3}{0}{3}
	\end{tikzpicture}}\\
 &\to
	{\begin{tikzpicture}[scale=\domscale+0.1,baseline=-20pt]
		\hobox{0}{0}{-6}
			\hobox{0}{1}{-4}
   \hobox{1}{1}{-2}
    \hobox{2}{1}{3}
            \hobox{1}{0}{-5}
            \hobox{2}{0}{-3}
            \hobox{3}{0}{2}
	\end{tikzpicture}}
 \to
	{\begin{tikzpicture}[scale=\domscale+0.1,baseline=-20pt]
		\hobox{0}{0}{-6}
			\hobox{0}{1}{-4}
   \hobox{1}{1}{-2}
    \hobox{2}{1}{3}
            \hobox{1}{0}{-5}
            \hobox{2}{0}{-3}
            \hobox{3}{0}{2}
             \hobox{4}{0}{5}
	\end{tikzpicture}}

  \to
	{\begin{tikzpicture}[scale=\domscale+0.1,baseline=-20pt]
		\hobox{0}{0}{-6}
			\hobox{0}{1}{-4}
   \hobox{1}{1}{-2}
    \hobox{2}{1}{3}
     \hobox{3}{1}{5}
            \hobox{1}{0}{-5}
            \hobox{2}{0}{-3}
            \hobox{3}{0}{2}
             \hobox{4}{0}{4}
	\end{tikzpicture}}
  
 \\
 & 
  
  \to
	{\begin{tikzpicture}[scale=\domscale+0.1,baseline=-20pt]
		\hobox{0}{0}{-6}
			\hobox{0}{1}{-4}
   \hobox{1}{1}{-2}
    \hobox{2}{1}{3}
     \hobox{3}{1}{5}
            \hobox{1}{0}{-5}
            \hobox{2}{0}{-3}
            \hobox{3}{0}{2}
             \hobox{4}{0}{4}
                \hobox{5}{0}{6}
	\end{tikzpicture}}
  =P(\lambda^-).
	\end{align*}

\end{example}

\begin{example}
 Let   $\mathfrak{g}=\mathfrak{so}(10,\mathbb{C})$.  Suppose $\Delta\setminus I=\{\alpha_{1},\alpha_{4}\}$. Then the corresponding   parabolic subalgebra 
  $\mathfrak{p}_I$ is a standard parabolic subalgebra of type  $(1,3,1)$. We can regard it as a standard parabolic subalgebra of type  $(1,4,0)$. By Theorem \ref{mainD}, a simple integral highest weight  module $L(\lambda)$  in $\mathcal{O}^{\mathfrak{p}}$  is socular  if and only if $P(\lambda^-)$ has the same even boxes as the following Z-diagram $P$ of type  $(0;1,4)$:  
\begin{center}
\begin{tikzpicture}[scale=\domscale-0.1,baseline=-15pt]
\hdom{0}{0}{B}
\hdom{0}{1}{B}
\hdom{0}{2}{B}
\hdom{0}{3}{B}
\hdom{2}{0}{B}
\end{tikzpicture}. 






\end{center}

The shape of $P(\lambda^-)\ev$ is $p(\lambda^-)\ev=p\ev=(2,1,1,1)$, where $p$ is the shape of the Z-diagram $P$.

Now  suppose $\lambda=(-9,-5,-6,-7,8)\in \mathfrak{h}^*$. Then $L(\lambda)$  in $\mathcal{O}^{\mathfrak{p}}$  is socular. We can check that $P(\lambda^-)$ has the same even boxes as the  Z-diagram $P$ of type  $(0; 1, 4)$. 
Actually,  since $\lambda^-=(-9,-5,-6,-7,8,-8,7,6,5,9)$ and from the R-S algorithm, we have

\begin{align*}
 &{\begin{tikzpicture}[scale=\domscale+0.1,baseline=-13pt]
			\hobox{0}{0}{-9}
			\hobox{1}{0}{-5}
			
	\end{tikzpicture}}\to
	{\begin{tikzpicture}[scale=\domscale+0.1,baseline=-28pt]
			\hobox{0}{0}{-9}
			\hobox{0}{1}{-6}
   \hobox{0}{2}{-5}
   \hobox{1}{0}{-7}
            \hobox{2}{0}{8}
           
	\end{tikzpicture}}\to
	{\begin{tikzpicture}[scale=\domscale+0.1,baseline=-38pt]
			\hobox{0}{0}{-9}
			\hobox{0}{1}{-7}
   \hobox{0}{2}{-6}
   \hobox{0}{3}{-5}
   \hobox{1}{0}{-8}
            \hobox{2}{0}{8}   
	\end{tikzpicture}}\to
	{\begin{tikzpicture}[scale=\domscale+0.1,baseline=-38pt]
			\hobox{0}{0}{-9}
			\hobox{0}{1}{-7}
   \hobox{0}{2}{-6}
   \hobox{0}{3}{-5}
   \hobox{1}{0}{-8}
    \hobox{1}{1}{8}   
    \hobox{2}{0}{7}   
	\end{tikzpicture}}\\
 &\to
	{\begin{tikzpicture}[scale=\domscale+0.1,baseline=-38pt]
		\hobox{0}{0}{-9}
			\hobox{0}{1}{-7}
   \hobox{0}{2}{-6}
   \hobox{0}{3}{-5}
   \hobox{1}{0}{-8}
    \hobox{1}{1}{7}  
      \hobox{1}{2}{8}
    \hobox{2}{0}{6}   
	\end{tikzpicture}}
 \to
	{\begin{tikzpicture}[scale=\domscale+0.1,baseline=-38pt]
		\hobox{0}{0}{-9}
			\hobox{0}{1}{-7}
   \hobox{0}{2}{-6}
   \hobox{0}{3}{-5}
   \hobox{1}{0}{-8}
    \hobox{1}{1}{6}  
      \hobox{1}{2}{7}
         \hobox{1}{3}{8}
    \hobox{2}{0}{5}   
	\end{tikzpicture}} \to
	{\begin{tikzpicture}[scale=\domscale+0.1,baseline=-38pt]
		\hobox{0}{0}{-9}
			\hobox{0}{1}{-7}
   \hobox{0}{2}{-6}
   \hobox{0}{3}{-5}
   \hobox{1}{0}{-8}
    \hobox{1}{1}{6}  
      \hobox{1}{2}{7}
         \hobox{1}{3}{8}
    \hobox{2}{0}{5}   
     \hobox{3}{0}{9}   
	\end{tikzpicture}}
  =P(\lambda^-).
	\end{align*}

\end{example}

\section{The non-integral case}\label{nonintegral}

Let $\mathfrak{g}=\mathfrak{sl}(n,\mathbb{C})$. Let $\frp$ be a parabolic subalgebra of size $(n_1, n_2,\cdots,n_k)$ with $\mathfrak{q}_{\mathfrak{p}}$ being the corresponding decreasing parabolic subalgebra of size $(m_1, m_2,\cdots,m_k)$. When a  weight  $\lambda \in \mathcal{O}^{\mathfrak{p}}$ is non-integral, from Bai-Xie \cite{BXie}, we can associate some Young tableaux  to $\lambda$. For any $ \lambda\in\mathfrak{h}^* $, we write $\lambda=(\lambda_1,\cdots,\lambda_n)$. Let   $ P(\lambda) $ be a set of some Young tableaux as follows. Let $ \lambda_Y: \lambda_{i_1}, \lambda_{i_2}, \dots, \lambda_{i_r} $  be a maximal subsequence of $ \lambda_1,\lambda_2,\dots,\lambda_n $ such that $ \lambda_{i_k} $, $ 1\leq k\leq r $ are congruent to each other by $ \mathbb{Z} $. Then let $P(\lambda_Y)$ be the Young tableau obtained from  $ \lambda_Y $ by using R-S algorithm, which is a Young tableau in $ P(\lambda) $.

 Now we put these Young tableaux together and make them into one bigger Young tableau $\bar{P}(\lambda)$ by inserting the columns of other Young tableaux into the first Young tableau such that $c_i(\bar{P}(\lambda))$ is decreasing for $1\leq i \leq k$. In other words,  the Young tableau $\bar{P}(\lambda)=\underset{P(\lambda_{Y})\in P(\lambda)}{\stackrel{c}{\sqcup}} P(\lambda_{Y}).$
Here $P_1{\stackrel{c}{\sqcup}}P_2$ denotes 	the Young tableau whose multiset of nonzero column lengths equals the union of	the two Young tableaux $P_1$ and $P_2$. 

 Then from Theorem  \ref{mainA}, we have the following theorem and corollary.

 \begin{theorem}Let $\mathfrak{g}=\mathfrak{sl}(n,\mathbb{C})$. Let $\frp$ be a parabolic subalgebra of type $(n_1, n_2,\cdots,n_k)$. A non-integral  highest weight module $L(\lambda)$  in $\mathcal{O}^{\mathfrak{p}}$  is socular  if and only if  $c_i(\bar{P}(\lambda))=m_i$ for $1\leq i \leq k$, where $(m_1,\cdots,m_k)$ is the  arrangement of the sequence $(n_1,n_2,..,n_{k})$ in descending order.
\end{theorem}

\begin{corollary}Let $\mathfrak{g}=\mathfrak{sl}(n,\mathbb{C})$. Let $\frp$ be a parabolic subalgebra of type $(n_1, n_2,\cdots,n_k)$. A non-integral  highest weight module $L(\lambda)$  in $\mathcal{O}^{\mathfrak{p}}$  is socular  if and only if we divide $\lambda=(\lambda_1,\cdots,\lambda_n)$ into several subsequences  such that any two entries of a subsequence has an integral difference and the restriction of  $\lambda$ to each subsequence is a socular weight in the corresponding parabolic category.

\end{corollary}

For the types $B_n, C_n$ and $D_n$, let
	$[\lambda] $  be the set of  maximal subsequences $ x $ of  $ \lambda $ such that any two entries of $ x $ have an integral  difference or sum. In this case, we set $ [\lambda]_1 $ (resp. $ [\lambda]_2 $) to be the subset of $ [\lambda] $ consisting of sequences with  all entries belonging to $ \mathbb{Z} $ (resp. $ \frac12+\mathbb{Z} $).
	We set $[\lambda]_{1,2}=[\lambda]_1\cup [\lambda]_2, \quad [\lambda]_3=[\lambda]\setminus[\lambda]_{1,2}$. 	Since there is at most one element in $[\lambda]_1 $ and $[\lambda]_2 $, we denote them by  $(\lambda)_{(0)}$ and $(\lambda)_{(\frac{1}{2})}$.	
	
	Let  $ x=(\lambda_{i_1}, \lambda_{i_2},\cdots \lambda_{i_r})\in[\lambda]_3 $. Let $  y=(\lambda_{j_1}, \lambda_{j_2},\cdots, \lambda_{j_p}) $ be the maximal subsequence of $ x $ such that $ j_1=i_1 $ and the difference of any two entries of $ y$ is an integer. Let $ z= (\lambda_{k_1}, \lambda_{k_2},\cdots, \lambda_{k_q}) $ be the subsequence obtained by deleting $ y$ from $ x $, which is possible empty.
	Define
	$$  \tilde{x}=(\lambda_{j_1}, \lambda_{j_2},\cdots, \lambda_{j_p}, -\lambda_{k_q}, -\lambda_{k_{q-1}},\cdots ,-\lambda_{k_1}).  $$

For $\lambda_Y\in (\lambda)_{(0)}\cup (\lambda)_{(\frac{1}{2})}$, we can get a Young tableau $P(\lambda_Y^-)$.
	
We define
\begin{align*}
	F_a(x)&=\sum_{j\geq 1}\frac{c_j(c_j-1)}{2}= \sum_{k\geq 1} (k-1) p_k,
	\end{align*}
where $ p(x)=(p_1,p_2,\cdots,p_N) $ is the shape of the Young tableau $P(x)$ and $p(x)^t=(c_1,\cdots,c_K)$.

\begin{proposition}[{\cite[Theorem 4.6]{BXie}} and {\cite[Theorem 5.7]{BXX}} ]\label{gkdim}
	The GK dimension of  $ L(\lambda) $  can be computed as follows.
	\begin{enumerate}
		\item If $  \mathfrak{g}= \mathfrak{sl}(n,\mathbb{C})$,
		\[
		\GK L(\lambda)=\frac{n(n-1)}{2}-\sum _{x\in [\lambda]} F_a(x).
		\]
		
		\item If $  \mathfrak{g}= \mathfrak{sp}(n,\mathbb{C})$,
		\[
		\GK L(\lambda)=n^2- F_b((\lambda)_{(0)}^-)- F_d((\lambda)_{(\frac{1}{2})}^-)-\sum _{x\in [\lambda]_3} F_a(\tilde{x}).
		\]
		\item  If $  \mathfrak{g} = \mathfrak{so}(2n+1,\mathbb{C}) $,
		\[
		\GK L(\lambda)=	n^2-F_b((\lambda)_{(0)}^-)-F_b((\lambda)_{(\frac{1}{2})}^-)-\sum _{x\in [\lambda]_3} F_a(\tilde{x}).
		\]
		\item  If $  \mathfrak{g} = \mathfrak{so}(2n,\mathbb{C}) $,
		\[
		\GK L(\lambda)=n^2-n-F_d((\lambda)_{(0)}^-)-F_d((\lambda)_{(\frac{1}{2})}^-)-\sum _{x\in [\lambda]_3} F_a(\tilde{x}).
		\]
	\end{enumerate}
\end{proposition}

From the above Proposition \ref{gkdim}, we can see that a non-integral  highest weight module $L(\lambda)$  in $\mathcal{O}^{\mathfrak{p}}$  is socular  if and only if all of the functions `$F$' in the summation of $\GK L(\lambda)$ take the minimal values.

Thus similar to type $A$, we have the following result.
\begin{corollary}
    Let $\mathfrak{g}=\mathfrak{so}(2n+1,\mathbb{C})$, $\mathfrak{sp}(n,\mathbb{C})$ or $\mathfrak{so}(2n+1,\mathbb{C})$. Let $\frp$ be a parabolic subalgebra of type $(n_1, n_2,\cdots,n_k)$. A non-integral  highest weight module $L(\lambda)$  in $\mathcal{O}^{\mathfrak{p}}$  is socular  if and only if $(\lambda)_{(0)}$, $(\lambda)_{(\frac{1}{2})}$ and all $\tilde{x}$ for $x\in [\lambda]_3$ are socular integral weights   in the corresponding parabolic categories.
\end{corollary}

\section{Richardson orbits associated to  parabolic subalgebras}
\label{richar}
In this section, we will give the partition type for the Richardson orbit $\mathcal{O}$ associated to a   standard parabolic subalgebra $\frp$  of type $(n_1, n_2,\cdots,n_k)$. We use ${\bf p}=[p_1,\cdots,p_k]$ to denote a partition of an integer $n$.

 First  we recall some results about nilpotent orbits of classical types. Some details can be found in \cite{CM93}.

\begin{proposition}[{\cite[Theorems 5.1.2, 5.1.3 $\&$ 5.1.4 and Proposition 6.3.7]{CM93}}]\label{bcd}
The nilpotent orbits of classical types can be identified with some partitions as follows:
\begin{enumerate}
    \item Nilpotent orbits in $\mathfrak{so}{(2n+1,\mathbb{C})}$ are in one-to-one correspondence with the set of partitions of $2n+1$ in which even parts occur with even multiplicity. A partition $\bf q$ of type $B$ is special if and only if its dual partition ${\bf q}^t$ is a partition of type $B$.
    \item Nilpotent orbits in $\mathfrak{sp}{(n,\mathbb{C})}$ are in one-to-one correspondence with the set of partitions of $2n$ in which odd parts occur with even multiplicity. A partition $\bf q$ of type $C$ is special if and only if its dual partition ${\bf q}^t$ is a partition of type $C$.
    \item Nilpotent orbits in $\mathfrak{so}{(2n,\mathbb{C})}$ are in one-to-one correspondence with the set of partitions of $2n$ in which even parts occur with even multiplicity,
  except that each ``very even" partition ${\bf d}$ (consisting of only even parts) correspond to two orbits, denoted by $\mathcal{O}^I_{\bf d}$ and $\mathcal{O}^{II}_{\bf d}$. A partition $\bf q$ of type $D$ is special if and only if its dual partition ${\bf q}^t$ is a partition of type $C$.
\end{enumerate}
\end{proposition}

\begin{proposition}[{\cite[Page 41]{Ja}} and {\cite[Page 118]{CM93}}]\label{special} The Richardson orbit $\mathcal{O}$ associated to   a  standard parabolic subalgebra $\frp$ of type $(n_1, n_2,\cdots,n_k)$ is unique and special. We have  $\dim \mathcal{O}=2\dim(\mathfrak{u})$. 
      
  \end{proposition}

\begin{theorem}
    Suppose  $\mathfrak{g}=\mathfrak{sl}(n,\mathbb{C})$. Let $\frp$ be a  standard parabolic subalgebra of type $(n_1, n_2,\cdots,n_k)$. The Richardson orbit $\mathcal{O}$ associated to $\frp $ has partition ${\bf p}=[m_1,\cdots,m_k]^t$, where $(m_1,\cdots,m_k)$ is the  arrangement of the sequence $(n_1,n_2,..,n_{k})$ in descending order.
\end{theorem}
\begin{proof}

    See \cite[Theorem 7.2.3]{CM93}.
    \end{proof}

\begin{theorem}\label{rich-b}
    Suppose  $\mathfrak{g}=\mathfrak{so}(2n+1,\mathbb{C})$. Let $\frp$ be a  standard parabolic subalgebra of type $(n_1, n_2,\cdots,n_k)$. The Richardson orbit $\mathcal{O}$ associated to $\frp $ has partition $${\bf p}=[p_1,p_2,\cdots,p_{_{2m_s}},p_{_{2m_s+1}}+1,p_{_{2m_s+2}},\cdots,p_N]_{_B},$$
    where $p=(p_1,\cdots,p_N)$ is the shape of the Z-diagram  $P$ of type  $(n_k; n_1,\cdots,n_{k-1})$ so that $${p}^t=(m_1,m_1,m_2,m_2,\cdots,m_{s-1},m_{s-1,},2m_s,m_{s+1},m_{s+1},\cdots,m_{k},m_{k}).$$ Here $(m_1,\cdots,m_k)$ is the  arrangement of the sequence $(n_1,n_2,..,n_{k})$ in descending order and $m_s=n_k$.
\end{theorem}

\begin{proof}
From Theorem \ref{mainBC}, an integral highest weight module $L(\lambda)$  in $\mathcal{O}^{\mathfrak{p}}$  is socular  if and only if $P(\lambda^{-})$ has the same odd boxes as a Z-diagram of type  $(n_k; n_1,\cdots,n_{k-1})$.   Recall that in this case, we can write $\lambda=w\mu$ for a unique anti-dominant $\mu\in \mathfrak{h}^*$ and a unique minimal length element $w\in W$. By \cite[Proposition 4.3]{I} or \cite[Theorem 48]{MCb}, $L(\lambda)$ in $\mathcal{O}^{\mathfrak{p}}$ is socular  if and only if $w$ belongs to  the Kazhdan-Lusztig right cell containing $w_0^{\mathfrak{p}}$, where $w_0^{\mathfrak{p}}$
 is the longest element in the parabolic subgroup
of $W$ corresponding to the parabolic subalgebra $\mathfrak{p}$. By Springer correspondence, we have $\mathcal{O}_w=\mathcal{O}_{w_0^{\mathfrak{p}}}$ being special, where $\mathcal{O}_{w_0^{\mathfrak{p}}}$ is the desired Richardson orbit. From \cite{BMW}, $\mathcal{O}_w$ has the same odd diagram as $P(\lambda^{-})$. Thus $\mathcal{O}_w$ has the same odd diagram as the Z-diagram $P$ of type  $(n_k; n_1,\cdots,n_{k-1})$. From the Z-diagram  $P$, by using H-algorithm, we can get the special partition ${\bf p}$ corresponding to $\mathcal{O}_w=\mathcal{O}_{w_0^{\mathfrak{p}}}$.
   
   When $n_k=0$, $P$ is a Young diagram with shape $p=(p_1,\cdots,p_N)$ such that $$p^t=(m_1,m_1,m_2,m_2,\cdots,m_{k-1},m_{k-1}).$$
Note that this $p$ contains only  even parts. After the H-algorithm on $p$, $p_1$ becomes $p_1+1$, some $p_{2i}$ becomes $p_{2i}-1$ and $p_{2i+1}$ becomes $p_{2i+1}+1$ (when $p_{2i}$ occurs with odd multiplicity and  $p_{2i}>p_{2i+1}$). Recall that a partition $\bf q$ of type $B$ is special if and only if its dual partition ${\bf q}^t$ is of type $B$.  Thus from Proposition \ref{bcd} and \ref{special} we have ${\bf p}=[p_1+1,p_2,\cdots,p_N]_{_B}$ since it is special and has the same odd diagram as the Z-diagram $P$.

   When $n_k>0$, $P$ is a Young diagram with shape $p=(p_1,\cdots,p_N)$ such that $$p^t=(m_1,m_1,m_2,m_2,\cdots,m_{s-1},m_{s-1,},2m_s,m_{s+1},m_{s+1},\cdots,m_{k},m_{k}),$$
   where $m_s=n_k$.
 After the H-algorithm on $p$, the first $2m_s$ parts will not change since they are odd parts. We only need to do $B$-collapse on the rest parts. Thus we have $${\bf p}=H(p)=[p_1,p_2,\cdots,p_{_{2m_s}},p_{_{2m_s+1}}+1,p_{_{2m_s+2}},\cdots,p_N]_{_B}.$$

Combined the above two cases, we complete the proof.
    \end{proof}
\begin{example}
 Let   $\mathfrak{g}=\mathfrak{so}(23,\mathbb{C})$.  Suppose $\Delta\setminus I=\{\alpha_{1},\alpha_{4},\alpha_{9}\}$. Then the corresponding   parabolic subalgebra 
  $\mathfrak{p}_I$ is a standard parabolic subalgebra of type  $(1,3,5,2)$. By Theorem \ref{mainBC}, a simple integral highest weight  module $L(\lambda)$  in $\mathcal{O}^{\mathfrak{p}}$  is socular  if and only if $P(\lambda^-)$ has the same odd boxes as the following Z-diagram $P$ of type  $(2;1,3,5)$:  
\begin{center}
\begin{tikzpicture}[scale=\domscale-0.1,baseline=-15pt]
\vdom{2}{0}{A}
\vdom{2}{2}{A}
\hdom{0}{0}{B}
\hdom{0}{1}{B}
\hdom{0}{2}{B}
\hdom{0}{3}{B}
\hdom{0}{4}{B}
\hdom{3}{0}{B}
\hdom{3}{1}{B}
\hdom{3}{2}{B}
\hdom{5}{0}{B}
\end{tikzpicture}. 
\end{center}

The shape of $P(\lambda^-)\od$ is $p(\lambda^-)\od=p\od=(3,3,2,2,1)$, where $p$ is the shape of the Z-diagram $P$.

From this Z-diagram $P$, we can get a Richardson orbit $\mathcal{O}$ with partition ${\bf p}=[p_1,p_2,\cdots,p_{_{2m_s}},p_{_{2m_s+1}}+1,p_{_{2m_s+2}},\cdots,p_N]_{_B}=[7,5,5,3,2+1]_{_B}=[7,5,5,3,3]$.

For example, suppose $\lambda=(-12,-9,-10,-11,-4,-5,-6,-7,-8,4,3)\in \mathfrak{h}^*$. Then $L(\lambda)$  in $\mathcal{O}^{\mathfrak{p}}$  is socular. We can check that $P(\lambda^-)$ has the same odd boxes as the  Z-diagram $P$ of type  $(2;1,3,5)$.

\end{example}

\begin{theorem}\label{rich-c}
    Suppose  $\mathfrak{g}=\mathfrak{sp}(n,\mathbb{C})$. Let $\frp$ be a  standard parabolic subalgebra of type $(n_1, n_2,\cdots,n_k)$. The Richardson orbit $\mathcal{O}$ associated to $\frp $ has partition $${\bf p}=[p_1,p_2,\cdots,p_N]_{_C},$$
    where $p=(p_1,\cdots,p_N)$ is the shape of the Z-diagram  $P$ of type  $(n_k; n_1,\cdots,n_{k-1})$ so that $${p}^t=(m_1,m_1,m_2,m_2,\cdots,m_{s-1},m_{s-1,},2m_s,m_{s+1},m_{s+1},\cdots,m_{k},m_{k}).$$ Here $(m_1,\cdots,m_k)$ is the  arrangement of the sequence $(n_1,n_2,..,n_{k})$ in descending order and $m_s=n_k$.
\end{theorem}
\begin{proof}
Similarly from the Z-diagram  $P$ of type  $(n_k; n_1,\cdots,n_{k-1})$, by using H-algorithm, we can get the special partition ${\bf p}$ corresponding to $\mathcal{O}_w=\mathcal{O}_{w_0^{\mathfrak{p}}}$.
   
   When $n_k=0$, $P$ is a Young diagram with shape $p=(p_1,\cdots,p_N)$ such that $$p^t=(m_1,m_1,m_2,m_2,\cdots,m_{k-1},m_{k-1}).$$
Note that this $p$ contains only  even parts, which is already a special partition of type $C$. Thus we have ${\bf p}=[p_1,p_2,\cdots,p_N]$.

   When $n_k>0$, $P$ is a Young diagram with shape $p=(p_1,\cdots,p_N)$ such that $$p^t=(m_1,m_1,m_2,m_2,\cdots,m_{s-1},m_{s-1,},2m_s,m_{s+1},m_{s+1},\cdots,m_{k},m_{k}),$$
   where $m_s=n_k$.
 After the H-algorithm on $p$, similar to type $B$ we have ${\bf p}=[p_1,p_2,\cdots,p_N]_{_C}$ since the last $m_1-2m_s$ parts are even parts and they will not change after H-algorithm.

Combined the above two cases, we complete the proof.
    \end{proof}

\begin{theorem}\label{rich-d}
    Suppose  $\mathfrak{g}=\mathfrak{so}(2n,\mathbb{C})$. Let $\frp$ be a  standard parabolic subalgebra of type $(n_1, n_2,\cdots,n_k)$. When $n_k\neq 1$, the Richardson orbit $\mathcal{O}$ associated to $\frp $ has partition $${\bf p}=[p_1,p_2,\cdots,p_N]_{_D},$$
    where $p=(p_1,\cdots,p_N)$ is the shape of the Z-diagram  $P$ of type  $(n_k; n_1,\cdots,n_{k-1})$ so that $${p}^t=(m_1,m_1,m_2,m_2,\cdots,m_{s-1},m_{s-1,},2m_s,m_{s+1},m_{s+1},\cdots,m_{k},m_{k}).$$ Here $(m_1,\cdots,m_k)$ is the  arrangement of the sequence $(n_1,n_2,..,n_{k})$ in descending order and $m_s=n_k$.

    When $n_k=1$, the Richardson orbit $\mathcal{O}$ associated to $\frp $ has partition $$\bar{\bf p}=[p'_1,p'_2,\cdots,p'_N]_{_D},$$
    where $\bar p=(p'_1,\cdots,p'_N)$ is the shape of the Z-diagram  $\bar P$ of type  $(0; n_1,\cdots,n_{k-2},n_{k-1}+1)$ so that $${\bar p}^t=(m_1,m_1,m_2,m_2,\cdots,m_{k-1},m_{k-1}).$$ Here $(m_1,\cdots,m_{k-1})$ is the  arrangement of the sequence $(n_1,n_2,..,n_{k-2},n_{k-1}+1)$ in descending order.
\end{theorem}
\begin{proof}
Similarly from the Z-diagram  $P$ of type  $(n_k; n_1,\cdots,n_{k-1})$, by using H-algorithm, we can get the special partition ${\bf p}$ corresponding to $\mathcal{O}_w=\mathcal{O}_{w_0^{\mathfrak{p}}}$.
In this case, $\mathcal{O}_w$ has the same even diagram as the Z-diagram $P$.

 Since the case of $n_k=1$ can be reduced to the case of $n_k=0$, we will not talk about it here. 
    
    When $n_k=0$, $P$ is a Young diagram with shape $p=(p_1,\cdots,p_N)$ such that $$p^t=(m_1,m_1,m_2,m_2,\cdots,m_{k-1},m_{k-1}).$$
Note that this $p$ contains only  even parts.  After the H-algorithm on $p$, similar to type $B$ we have ${\bf p}=[p_1,p_2,\cdots,p_N]_{_D}=p_{_D}$.

   When $n_k>1$, $P$ is a Young diagram with shape $p=(p_1,\cdots,p_N)$ such that $$p^t=(m_1,m_1,m_2,m_2,\cdots,m_{s-1},m_{s-1,},2m_s,m_{s+1},m_{s+1},\cdots,m_{k},m_{k}),$$
   where $m_s=n_k$.
 After the H-algorithm on $p$, we have $${\bf p}=H(p)=[p_1,p_2,\cdots,p_N]_{_D}=p_{_D}$$ since the first $2m_s$ parts are odd parts and they will not change after the H-algorithm.

Combined the above two cases, we complete the proof.
    \end{proof}

\begin{remark}
   When the associated Richardson orbit is a very even orbit, its numeral can be determined by  \cite[Theorem 7.3.3(ii)]{CM93} since $n_k=0$ or $1$ in this case. 
\end{remark}
\begin{example}
 Let   $\mathfrak{g}=\mathfrak{so}(22,\mathbb{C})$.  Suppose $\Delta\setminus I=\{\alpha_{1},\alpha_{8}\}$. Then the corresponding   parabolic subalgebra 
  $\mathfrak{p}_I$ is a standard parabolic subalgebra of type  $(1,7,3)$. By Theorem \ref{mainD}, a simple integral highest weight  module $L(\lambda)$  in $\mathcal{O}^{\mathfrak{p}}$  is socular  if and only if $P(\lambda^-)$ has the same even even boxes as the following Z-diagram $P$ of type  $(3;1,7)$:  
 
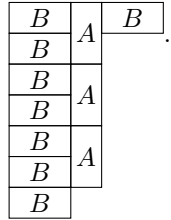
\begin{center}   
\begin{tikzpicture}[scale=\domscale-0.1,baseline=-15pt]
\hdom{0}{0}{B}
\hdom{0}{1}{B}
\hdom{0}{2}{B}
\hdom{0}{3}{B}
\hdom{0}{4}{B}
\hdom{0}{5}{B}
\hdom{0}{6}{B}
\vdom{2}{0}{A}
\vdom{2}{2}{A}
\vdom{2}{4}{A}
\hdom{3}{0}{B}
\end{tikzpicture}. 
\end{center} 

The shape of $P(\lambda^-)\ev$ is $p(\lambda^-)\ev=p\ev=(3,1,2,1,2,1,1)$, where $p$ is the shape of the Z-diagram $P$.

From this Z-diagram $P$, we can get a Richardson orbit $\mathcal{O}$ with partition ${\bf p}=[p_1,p_2,\cdots,p_N]_{_D}=[5,3^5,2]_{_D}=[5,3^5,1,1]$.

For example, suppose $\lambda=(-20,-5,-6,-7,-8,-9,-10,-11,3,2,1)\in \mathfrak{h}^*$. Then $L(\lambda)$  in $\mathcal{O}^{\mathfrak{p}}$  is socular. We can check that $P(\lambda^-)$ has the same even boxes as the  Z-diagram $P$ of type  $(3;1,7)$.
\end{example}

\subsection*{Acknowledgments}
This work was supported by National Natural Science Foundation of China  (Grant No. 12171344) and the National Key $\textrm{R}\,\&\,\textrm{D}$ Program of China (No. 2018YFA0701700 and No. 2018YFA0701701). The authors thank Volodymyr Mazorchuk for very helpful discussions about the socular highest weight modules in the case of type $A$. 
 The authors thank the referees for the careful reading and comments.




\begin{thebibliography}{99}
\bibitem{BMW}Bai Z Q,  Ma J -J,  Wang Y T. A combinatorial characterization of the   annihilator varieties of highest weight modules for  classical Lie algebras. arXiv:2304.03475, 2023


 \bibitem{BXiao}  Bai Z Q,  Xiao W.  Gelfand-Kirillov dimension and reducibility of scalar generalized Verma modules. Acta Math Sin (Engl Ser), 2019, 35:  1854-1860		
 
 
 
 \bibitem{BXX} Bai Z Q, Xiao W,  Xie X. Gelfand-Kirillov dimensions and associated varieties of highest weight modules. Int Math Res Not IMRN, 2023, 2023: 8101–8142




\bibitem{BXie} Bai Z Q, Xie X. Gelfand-Kirillov dimensions of highest weight Harish-Chandra modules for $SU(p,q)$. Int Math Res Not IMRN, 2019, 2019:  4392-4418

\bibitem{Ba} Baur K. Richardson elements for classical Lie algebras. J Algebra, 2006,  297:  168-185








\bibitem{BK} Brundan J, Kleshchev A.  Schur-Weyl duality for higher levels. Selecta Math, 2008, 14:  1-57



\bibitem{BLW} Brundan J, Losev I, Webster B.  Tensor product categorifications and the super Kazhdan-Lusztig conjecture. Int Math Res Not IMRN, 2017, 2017: 6329-6410

\bibitem{Br}  Brundan J. M{\oe}glin's theorem and Goldie rank polynomials in Cartan type $A$. Compos Math,  2011, 147: 1741-1771

\bibitem{CM93} Collingwood D J,   McGovern W M. Nilpotent orbits in semisimple Lie algebras.
Van Nostrand Reinhold Mathematics Series. New York: Van Nostrand Reinhold Co., 
1993

\bibitem{Gaa}
Garfinkle D. On the classification of primitive ideals for complex classical Lie algebras. I. Compos Math, 1990, 75:  135-169

\bibitem{Gab}
 Garfinkle D. On the classification of primitive ideals for complex classical Lie algebra. II. Compos Math, 1992, 81:  307-336

\bibitem{Gac}
Garfinkle D. On the classification of primitive ideals for complex classical Lie algebras. III. Compos Math, 1993, 88:  187-234

\bibitem{He}
Hesselink W H.  Polarizations in the classical groups. Math Z, 1978, 160: 217-234

\bibitem{HR}
Hille L,  R\"{o}hrle G. A classification of parabolic subgroups of classical groups with a finite number of orbits on the unipotent radical. Transform Groups,  1999, 4: 35-52

\bibitem{I}  Irving R S.  Projective modules in the category $\mathcal{O}_S$: self-duality.
Trans Amer Math Soc, 1985, 291:  701-732


\bibitem{Ja} 
 Jantzen J C. Nilpotent orbits in representation theory. In Lie theory: Lie Algebras and Representations. Progr Math, vol. 288. Boston: Birkh\"{a}user Boston, Inc.,  MA, 2004,  1-211



\bibitem{Lus79} Lusztig G.  A class of irreducible representations of a Weyl group. Nederl Akad Wetensch Indag Math, 1979, 41: 323-335


\bibitem{Lus84} Lusztig G. Characters of reductive groups over a finite field.   Annals of Math Studies, vol. 107. Princeton: Princeton University Press,  1984






\bibitem{MC} Mazorchuk V,  Stroppel C. Projective-injective modules, Serre functors and symmetric algebras.  J Reine Angew Math, 2008, 616:  131-165


\bibitem{MCb} Mazorchuk V,  Stroppel C.
Categorification of (induced) cell modules and the rough structure of generalised Verma modules. 
Adv Math, 2008, 219: 1363–1426

\bibitem{Pa} Panyushev D I. On spherical nilpotent orbits and beyond. Ann Inst Fourier (Grenoble), 1999, 49: 1453-1476

\bibitem{Ri}
Richardson R W.   Conjugacy classes in parabolic subgroups of semisimple algebraic groups. Bull Lond Math Soc, 1974, 6:  21–24 



\bibitem{Sa} Sagan  B E. The Symmetric Group: Representations, Combinatorial Algorithms, and Symmetric Functions.  Graduate Texts in Mathematics, vol. 203.  New York: Springer-Verlag,  2001





\bibitem{V} Vogan D A. Gelfand-Kirillov dimension for Harish-Chandra modules.  Invent Math, 1978, 48:  75-98



\bibitem{X} Xiao W. Leading weight vectors and homomorphisms between generalized Verma modules. J Algebra, 2015,  430: 62-93




\end{thebibliography}
\end{document}